\documentclass[twoside,11pt]{article}
\usepackage{a4wide,cite,latexsym,amsfonts,amssymb,exscale,epsfig,hyperref,url}
\usepackage[centertags,sumlimits,intlimits,namelimits,reqno]{amsmath}

\usepackage{amsthm}
\theoremstyle{definition}
\newtheorem{theorem}{Theorem}[section]
\newtheorem{lemma}[theorem]{Lemma}
\newtheorem{corollary}[theorem]{Corollary}
\newtheorem{proposition}[theorem]{Proposition}
\newtheorem{definition}[theorem]{Definition}
\newtheorem{remark}[theorem]{Remark}
\newtheorem{example}[theorem]{Example}

\usepackage{bbm}
\def\C{{\mathbbm C}}
\def\N{{\mathbbm N}}

\def\Z{{\mathbbm Z}}

\def\ssl{{\mathfrak{sl}}}

\def\ie{{\sl i.e.\/}}

\def\cf{{\sl c.f.\/}}

\let\phi=\varphi
\let\theta=\vartheta
\let\epsilon=\varepsilon

\def\ev{\mathop{\rm ev}\nolimits}
\def\coev{\mathop{\rm coev}\nolimits}

\def\dim{\mathop{\rm dim}\nolimits}

\def\id{\mathop{\rm id}\nolimits}

\def\End{\mathop{\rm End}\nolimits}

\def\Hom{\mathop{\rm Hom}\nolimits}

\def\ker{\mathop{\rm ker}\nolimits}
\def\im{\mathop{\rm im}\nolimits}
\def\coim{\mathop{\rm coim}\limits}

\def\lim{\mathop{\rm lim}\limits}

\let\hat=\widehat
\let\tilde=\widetilde

\def\pprime{{\prime\prime}}
\def\ppprime{{\prime\prime\prime}}
\def\hotimes{{\hat\otimes}}

\def\q#1{{[#1]}_q}

\numberwithin{equation}{section}

\makeatletter

\newfont{\@aidxte}{cmsy10}
\newfont{\@aidxel}{cmsy10 scaled 1095}
\newfont{\@aidxtw}{cmsy10 scaled 1200}
\newlength\@aidxtexvi
\newlength\@aidxtexvii
\newlength\@aidxelxvi
\newlength\@aidxelxvii
\newlength\@aidxtwxvi
\newlength\@aidxtwxvii
\newcommand{\alignidx}[1]{%
\@aidxtexvi=\fontdimen16\@aidxte
\@aidxtexvii=\fontdimen17\@aidxte
\@aidxelxvi=\fontdimen16\@aidxel
\@aidxelxvii=\fontdimen17\@aidxel
\@aidxtwxvi=\fontdimen16\@aidxtw
\@aidxtwxvii=\fontdimen17\@aidxtw
{\mbox{$%
\fontdimen16\@aidxte=2.9pt
\fontdimen17\@aidxte=2.9pt
\fontdimen16\@aidxel=3.1pt
\fontdimen17\@aidxel=3.1pt
\fontdimen16\@aidxtw=3.3pt
\fontdimen17\@aidxtw=3.3pt
#1$}}%
\fontdimen16\@aidxte=\@aidxtexvi
\fontdimen17\@aidxte=\@aidxtexvii
\fontdimen16\@aidxel=\@aidxelxvi
\fontdimen17\@aidxel=\@aidxelxvii
\fontdimen16\@aidxtw=\@aidxtwxvi
\fontdimen17\@aidxtw=\@aidxtwxvii}

\makeatother

\newenvironment{myenumerate}{%
\begin{enumerate}
\setlength{\partopsep}{0pt}
\setlength{\parskip}{0pt}}{\end{enumerate}}

\newenvironment{myitemize}{%
\begin{itemize}
\setlength{\itemsep}{0pt}
\setlength{\parskip}{0pt}}{\end{itemize}}

\newcommand{\ontop}[2]{\genfrac{}{}{0pt}{2}{\scriptstyle #1}{\scriptstyle
#2}}
\def\nn{\notag}

\def\emph#1{{\sl #1\/}}
\def\sym#1{{\mathcal #1}}
\def\one{\mathbbm{1}}%
\def\bar#1{\overline{#1}}%

\def\op{\mathrm{op}}

\def\Vect{\mathbf{Vect}}
\def\fdVect{\mathbf{fdVect}}

\def\coend{\mathbf{coend}}

\def\msc#1{\noindent{\small Mathematics Subject Classification (2000):
#1\par}}%
\def\keywords#1{\noindent {\small keywords: #1\par}}%

\usepackage[ps,matrix,arrow,frame,curve]{xy}
\everyxy={0;<1mm,0mm>:}

\begin{document}

\title{Fusion categories in terms of graphs and relations}
\author{Hendryk Pfeiffer\thanks{E-mail: \texttt{pfeiffer@math.ubc.ca}}}
\date{\small{Department of Mathematics, The University of British Columbia,\\
1984 Mathematics Road, Vancouver, BC, V2T 1Z2, Canada}\\[1ex]
April 13, 2011}

\maketitle

\begin{abstract}

Every fusion category $\sym{C}$ that is $k$-linear over a suitable
field $k$, is the category of finite-dimensional comodules of a Weak
Hopf Algebra $H$. This Weak Hopf Algebra is finite-dimensional,
cosemisimple and has commutative bases. It arises as the universal
coend with respect to the long canonical functor
$\omega\colon\sym{C}\to\Vect_k$. We show that $H$ is a quotient
$H=H[\sym{G}]/I$ of a Weak Bialgebra $H[\sym{G}]$ which has a
combinatorial description in terms of a finite directed graph
$\sym{G}$ that depends on the choice of a generator $M$ of $\sym{C}$
and on the fusion coefficients of $\sym{C}$. The algebra underlying
$H[\sym{G}]$ is the path algebra of the quiver
$\sym{G}\times\sym{G}$, and so the composability of paths in
$\sym{G}$ parameterizes the truncation of the tensor product of
$\sym{C}$. The ideal $I$ is generated by two types of relations. The
first type enforces that the tensor powers of the generator $M$ have
the appropriate endomorphism algebras, thus providing a Schur--Weyl
dual description of $\sym{C}$. If $\sym{C}$ is braided, this
includes relations of the form `$RTT=TTR$' where $R$ contains the
coefficients of the braiding on $\omega M\otimes\omega M$, a
generalization of the construction of
Faddeev--Reshetikhin--Takhtajan to Weak Bialgebras. The second type
of relations removes a suitable set of group-like elements in order
to make the category of finite-dimensional comodules equivalent to
$\sym{C}$ over all tensor powers of the generator $M$. As examples,
we treat the modular categories associated with
$U_q(\mathfrak{sl}_2)$.
\end{abstract}

\msc{%
16W30,
18D10
}

\section{Introduction}

A fusion category $\sym{C}$ is a semisimple, additive, rigid monoidal
category. $\sym{C}$ is required to have only a finite number of simple objects
up to isomorphism (we call this \emph{finitely semisimple}) and to be
$k$-linear over some field $k$ such that $\Hom(X,Y)$ is finite-dimensional for
all objects $X,Y\in|\sym{C}|$. In the following, we do not impose any
condition on the field $k$, but we require that $\End(X)\cong k$ for all
simple objects $X\in|\sym{C}|$ and say that $\sym{C}$ is \emph{split
semisimple}. For technical reasons, $\sym{C}$ is required to be essentially
small, and for convenience, we equip every object $X\in|\sym{C}|$ with a
specified left-dual. Such a rigid category is called
\emph{left-autonomous}. We do not require the monoidal unit $\one$ to be
simple, \ie\ we include the case of multi-fusion categories. For further
background on fusion categories, we refer
to~\cite{Tu94,BaKi01,EtNi05}.

If a fusion category $\sym{C}$ arises as the category of comodules
$\sym{M}^H\simeq\sym{C}$ of some Hopf algebra $H$, it admits a functor
$F\colon\sym{C}\to\Vect_k$ that is strong monoidal, \ie\ in particular
$F(X\otimes Y)\cong FX\otimes_k FY$ are isomorphic vector spaces for
all $X,Y\in|\sym{C}|$. This happens because for every Hopf algebra
$H$, the forgetful functor $\sym{M}^H\to\Vect_k$ is strong
monoidal. The most interesting fusion categories are those that do
\emph{not} admit any strong monoidal functor to $\Vect_k$ and
therefore do \emph{not} arise from Hopf algebras in this way.

Nevertheless, each fusion category $\sym{C}$ still admits the \emph{long
canonical functor}
\begin{eqnarray}
\omega\colon\sym{C}\to\Vect_k,\quad X &\mapsto&\Hom(\hat V,\hat V\otimes X),\\
f &\mapsto& (\id_{\hat V}\otimes f)\circ-.\nn
\end{eqnarray}
Here, we have used the small progenerator
\begin{equation}
\hat V=\bigoplus_{j\in I} V_j,
\end{equation}
where ${\{V_j\}}_{j\in I}$ is a set of representatives of the
isomorphism classes of the simple objects of $\sym{C}$. The functor
$\omega$ is $k$-linear, faithful, exact, and has a \emph{separable
Frobenius structure}~\cite{Ha99b,Sz05,Pf09a} which includes the
structure of both a lax and an oplax monoidal functor.

Under this functor, the $k$-dimension of a tensor product
\begin{equation}
\dim_k\omega(X\otimes Y)\leq(\dim_k\omega X)\cdot(\dim_k\omega Y),
\end{equation}
is in general smaller than the product of $k$-dimensions. This effect is known
as the \emph{truncation} of the tensor product.

In the present article, we use the long canonical functor in order to
arrive at a characterization of fusion categories in which we can
fully parameterize the truncation of the tensor product in terms of
combinatorial data. This is done as follows.

First, Tannaka--Kre\v\i n reconstruction has been generalized to
functors with a separable Frobenius
structure~\cite{Ha99b,Os03a,Pf09a,Mc09}. It equips the universal coend
$H=\coend(\sym{C},\omega)$ with the structure of a Weak Hopf Algebra
(WHA)~\cite{BoNi99,BoSz00}. This WHA can be shown to be
finite-dimensional and split cosemisimple and to have commutative
bases~\cite{Pf09a}. The long canonical functor plays the role of the
forgetful functor $\sym{M}^H\to\Vect_k$ of the category $\sym{M}^H$ of
finite-dimensional right $H$-comodules, and $\sym{C}\simeq\sym{M}^H$
are equivalent as $k$-linear additive monoidal categories.

We then choose an object $M\in|\sym{C}|$ that generates $\sym{C}$ as a
fusion category, and define a finite directed graph $\sym{G}$, the
\emph{dimension graph of $\sym{C}$ with respect to $M$}. It depends on
the choice of the generator $M$ and on the fusion coefficients of
$\sym{C}$. The reason for considering this graph is the following.

The algebra $R:=\End(\hat V)\cong k^{|I|}$ has a basis of orthogonal
idempotents $\lambda_j=\id_{V_j}$, $j\in I$. The vector spaces $\omega
X=\Hom (\hat V,\hat V\otimes X)$ form $R$-$R$-bimodules. We can choose
a basis of $\omega X$ that consists of basis vectors of the
$\Hom(V_j,V_\ell\otimes X)$ for all $j,\ell\in I$, \ie\ one that is
adapted to the orthogonal idempotents of $R$.

The tensor product in $\sym{C}\simeq\sym{M}^H$ is governed by the
multiplication in $H$ which, in turn, is determined by the lax and
oplax monoidal structure of $\omega$. The following is the lax
monoidal structure:
\begin{gather}
\omega_0\colon
k \to \omega\one,\qquad
1 \mapsto \rho_{\hat V}^{-1},\\
\omega_{X,Y}\colon
\omega X\otimes\omega Y \to \omega(X\otimes Y),\qquad
f\otimes g \mapsto \alpha_{\hat V,X,Y}\circ(f\otimes\id_Y)\circ g.
\end{gather}
It is not difficult to see that $\omega(X\otimes Y)\cong \omega
X\otimes_R\omega Y$ is the tensor product in the category of
$R$-$R$-bimodules. In particular, if $f_1\in\Hom (V_j,V_\ell\otimes M)$ and
$f_2\in\Hom (V_p,V_q\otimes M)$, $j,\ell,p,q\in I$, then
$\omega_{M,M}(f_1\otimes f_2)$ is non-zero if and only if $q=j$.

We thus define the \emph{dimension graph} $\sym{G}$ of $\sym{C}$ with
respect to $M$ to have vertices $\sym{G}^0=I$, \ie\ the orthogonal
idempotents of $R$, and edges $\sym{G}^1_{\ell j}$ from $j$ to $\ell$
the vectors of a basis of $\Hom(V_j,V_\ell\otimes M)$. Then, two edges
$f_1,f_2\in\sym{G}^1$ are composable if and only if the truncated
tensor product $\omega(M\otimes M)$ contains the corresponding vector
$f_1\otimes f_2$ of the $k$-linear tensor product $\omega
M\otimes_k\omega M$.

There is a Weak Bialgebra (WBA) $H[\sym{G}]$ associated with the graph
$\sym{G}$ and a surjection of WBAs $\pi\colon H[\sym{G}]\to H$ such that a
specific simple comodule $k\sym{G}^1$ of $H[\sym{G}]$ is pushed forward under
$\pi$ to the generating comodule $\omega M$ of $\sym{M}^H$. Since $\pi$ is a
homomorphism of WBAs, the same holds for all tensor powers thereof, \ie\ that
${(k\sym{G}^1)}^{\hotimes m}$ is pushed forward to ${(\omega M)}^{\hotimes
m}$, $m\geq 0$.

In order to characterize $H$ and thereby the fusion category
$\sym{C}\simeq\sym{M}^H$, it remains to compute the kernel of $\pi$. This is
done in two steps.

First, we take a suitable quotient $H[\sym{G},\sym{E}]:=H[\sym{G}]/I_\sym{E}$
in order to enforce that each ${(k\sym{G}^1)}^{\hotimes m}$, $m\geq 0$, is
equipped with the same endomorphism algebra as ${(\omega M)}^{\hotimes m}$.

This quotient is particularly easy if the monoidal unit $\one$ and the
chosen generator $M$ of $\sym{C}$ are both simple and if $\sym{C}$ is
braided such that braiding and inverse braiding of adjacent tensor
factors already generate all endomorphisms of ${(\omega M)}^{\hotimes
m}$, $m\geq 2$. In this case, the ideal $I_\sym{E}$ is generated by
quadratic relations of the form `$RTT=TTR$', a generalization of the
construction of Faddeev--Reshetikhin--Takhtajan (FRT)~\cite{ReTa90} to
WBAs. The coefficients of the $R$-matrix in these relations form the
Boltzmann weight of a star-triangular face model.

Second, our surjection of WBAs factors through this first quotient
yielding another surjection of WBAs $\bar\pi\colon
H[\sym{G},\sym{E}]\to H$. We show that $\ker\bar\pi$ is generated by
$1-g$ for a suitable subset of group-like elements $g\in
H[\sym{G},\sym{E}]$. Dividing by $1-g$ ensures that the categories of
comodules of $H$ and $H[\sym{G},\sym{E}]/\ker\pi$ agree everywhere,
not just for fixed tensor powers of the generator $M$.

For the special case of $U_q(\mathfrak{sl}_N)$ in which the first
quotient is given in terms of $RTT$ relations, Hayashi~\cite{Ha99a}
has already presented WHAs whose categories of finite-dimensional
comodules have the same fusion rules as the modular categories
associated with $U_q(\mathfrak{sl}_N)$. In fact, in this special case
of our construction, the first quotient $H[\sym{G}]/I_\sym{E}$ appears
in the literature on subfactors, see, for example~\cite{Oc88}. In
Ocneanu's terminology, the Weak Bialgebra $H[\sym{G}]$ is called a
\emph{paragroup} and the coefficients of the $R$-matrix a
\emph{connection}. The original FRT construction was reformulated by
M{\"u}ller~\cite{Mu01} in a way that can be directly compared with our
approach.

The present article is organized as follows. Section~\ref{sect_prelim}
summarizes some background material on WBAs and WHAs and on the
generalization of Tannaka--Kre\v\i n reconstruction to our case. In
Section~\ref{sect_combinatorial}, we construct the dimension graph
$\sym{G}$ and the surjection of WBAs $\pi\colon H[\sym{G}]\to H$. The
first quotient $H[\sym{G},\sym{E}]=H[\sym{G}]/I_\sym{E}$ is studied in
Section~\ref{sect_schurweyl}. In Section~\ref{sect_monoid}, we study
the group-like elements of the WBA $H[\sym{G},\sym{E}]$
and their associated comodules in order to compute the kernel of
$\bar\pi\colon H[\sym{G},\sym{E}]\to H$. As examples, the modular
categories associated with $U_q(\mathfrak{sl}_2)$
are treated in Section~\ref{sect_example}. The reader who is
interested in a quick overview of our construction, is encouraged to
go directly to that section. Appendix~\ref{app_monoidal} contains a
summary of the definitions and conventions for monoidal categories
that we use.

\section{Preliminaries}
\label{sect_prelim}

In Subsection~\ref{sect_prelimwha}, we summarize some key definitions and
properties of Weak Bialgebras (WBAs) and Weak Hopf Algebras (WHAs)
following~\cite{BoNi99,BoSz00}. Subsection~\ref{sect_prelimcomod} reviews
their categories of comodules following~\cite{Pf09a}. In
Subsection~\ref{sect_prelimtk}, we recall the main results about the
Tannaka--Kre\v\i n reconstruction of a WHA from a given monoidal
category~\cite{Pf09a,Mc09}.

We use the following notation. If $\sym{C}$ is a category, we write
$X\in|\sym{C}|$ for the objects $X$ of $\sym{C}$, $\Hom(X,Y)$ for the
collection of all morphisms $f\colon X\to Y$ and $\End(X)=\Hom(X,X)$. We
denote the identity morphism of $X$ by $\id_X\colon X\to X$ and the
composition of morphisms $f\colon X\to Y$ and $g\colon Y\to Z$ by $g\circ
f\colon X\to Z$. If two objects $X,Y\in|\sym{C}|$ are isomorphic, we write
$X\cong Y$. If two categories $\sym{C}$ and $\sym{D}$ are equivalent, we write
$\sym{C}\simeq\sym{D}$. The identity functor on $\sym{C}$ is denoted by
$1_{\sym{C}}$. The category of vector spaces over a field $k$ is denoted by
$\Vect_k$ and its full subcategory of finite-dimensional vector spaces by
$\fdVect_k$. Both are $k$-linear, abelian and symmetric monoidal. The $n$-fold
tensor power of some object $X\in|\sym{C}|$ of a monoidal category
$(\sym{C},\otimes,\one,\alpha,\lambda,\rho)$ is denoted by $X^{\otimes n}$,
$n\in\N_0$. We set $X^{\otimes 0}:=\one$. We use the notation $\N$
and $\N_0$ for the positive integers and the non-negative integers,
respectively. For our notation and conventions regarding monoidal categories
with duals as well as additive and abelian categories, we refer to
Appendix~\ref{app_monoidal}.

\subsection{Weak Hopf Algebras}
\label{sect_prelimwha}

\begin{definition}
A \emph{Weak Bialgebra} $(H,\mu,\eta,\Delta,\epsilon)$ over a field
$k$ is a $k$-vector space $H$ such that
\begin{myenumerate}
\item
$(H,\mu,\eta)$ forms an associative algebra with multiplication $\mu\colon
H\otimes H\to H$ and unit $\eta\colon k\to H$,
\item
$(H,\Delta,\epsilon)$ forms a coassociative coalgebra with comultiplication
$\Delta\colon H\to H\otimes H$ and counit $\epsilon\colon H\to k$,
\item
the following compatibility conditions hold:
\begin{eqnarray}
\Delta\circ\mu
&=& (\mu\otimes\mu)\circ(\id_H\otimes\sigma_{H,H}\otimes\id_H)\circ(\Delta\otimes\Delta),\\
\epsilon\circ\mu\circ(\mu\otimes\id_H)
&=& (\epsilon\otimes\epsilon)\circ(\mu\otimes\mu)\circ(\id_H\otimes\Delta\otimes\id_H)\nn\\
&=& (\epsilon\otimes\epsilon)\circ(\mu\otimes\mu)\circ(\id_H\otimes\Delta^\op\otimes\id_H),\\
(\Delta\otimes\id_H)\circ\Delta\circ\eta
&=& (\id_H\otimes\mu\otimes\id_H)\circ(\Delta\otimes\Delta)\circ(\eta\otimes\eta)\nn\\
&=& (\id_H\otimes\mu^\op\otimes\id_H)\circ(\Delta\otimes\Delta)\circ(\eta\otimes\eta).
\end{eqnarray}
\end{myenumerate}
Here $\sigma_{V,W}\colon V\otimes W\to W\otimes V$, $v\otimes w\mapsto
w\otimes v$ is the transposition of the tensor factors, and by
$\Delta^\op=\sigma_{H,H}\circ\Delta$ and $\mu^\op=\mu\circ\sigma_{H,H}$, we
denote the \emph{opposite comultiplication} and \emph{opposite
multiplication}, respectively. We tacitly identify the vector spaces
$(V\otimes W)\otimes U\cong V\otimes(W\otimes U)$ and $V\otimes k\cong V\cong
k\otimes V$, exploiting the coherence theorem for the monoidal category
$\Vect_k$.

A \emph{homomorphism} $\phi\colon H\to H^\prime$ of WBAs over the same field
$k$ is a $k$-linear map that is a homomorphism of unital algebras as well as a
homomorphism of counital coalgebras.
\end{definition}

In a WBA $H$, there are two important linear idempotents, the
\emph{source counital map}
\begin{equation}
\epsilon_s:=(\id_H\otimes\epsilon)\circ(\id_H\otimes\mu)\circ(\sigma_{H,H}\otimes\id_H)
\circ(\id_H\otimes\Delta)\circ(\id_H\otimes\eta)\colon H\to H
\end{equation}
and the \emph{target counital map}
\begin{equation}
\epsilon_t:=(\epsilon\otimes\id_H)\circ(\mu\otimes\id_H)\circ(\id_H\otimes\sigma_{H,H})
\circ(\Delta\otimes\id_H)\circ(\eta\otimes\id_H)\colon H\to H.
\end{equation}
Their images $H_s:=\epsilon_s(H)$ and $H_t:=\epsilon_t(H)$ are
mutually commuting unital subalgebras and are called the \emph{source
base algebra} and the \emph{target base algebra}, respectively.

\begin{definition}
A \emph{Weak Hopf Algebra} $(H,\mu,\eta,\Delta,\epsilon,S)$ is a Weak
Bialgebra $(H,\mu,\eta,\Delta,\epsilon)$ with a linear map $S\colon
H\to H$ (\emph{antipode}) that satisfies the following conditions:
\begin{eqnarray}
\mu\circ(\id_H\otimes S)\circ\Delta &=& \epsilon_t,\\
\mu\circ(S\otimes\id_H)\circ\Delta &=& \epsilon_s,\\
\mu\circ(\mu\otimes\id_H)\circ(S\otimes\id_H\otimes S)
\circ(\Delta\otimes\id_H)\circ\Delta&=&S.
\end{eqnarray}
Note that if $f\colon H\to H^\prime$ is a homomorphism of WBAs and both
$H$ and $H^\prime$ are WHAs, then $S^\prime\circ f=f\circ S$.
\end{definition}

For convenience, we write $1=\eta(1)$ and omit parentheses in products,
exploiting associativity. We also use Sweedler's notation and write
$\Delta(x)=x^\prime\otimes x^\pprime$ for the comultiplication of $x\in H$ as
an abbreviation of the expression $\Delta(x)=\sum_k a_k\otimes b_k$ with some
$a_k,b_k\in H$. Similarly, we write
$((\Delta\otimes\id_H)\circ\Delta)(x)=x^\prime\otimes x^\pprime\otimes
x^\ppprime$, exploiting coassociativity.

\begin{definition}
\label{def_coquasi}
A \emph{coquasitriangular} WHA $(H,\mu,\eta,\Delta,\epsilon,S,r)$ over
a field $k$ is a WHA $(H,\mu,\eta,\Delta,\epsilon,S)$ over $k$ with a
linear form $r\colon H\otimes H\to k$ (\emph{universal $r$-form}) that
satisfies the following conditions:
\begin{myenumerate}
\item
For all $x,y\in H$,
\begin{equation}
r(x\otimes y)=\epsilon(x^\prime y^\prime)r(x^\pprime\otimes y^\pprime)
=r(x^\prime\otimes y^\prime)\epsilon(y^\pprime x^\pprime).
\end{equation}
\item
There exists a linear form $\bar r\colon H\otimes H\to k$ such that for all
$x,y\in H$,
\begin{eqnarray}
\bar r(x^\prime\otimes y^\prime)r(x^\pprime\otimes y^\pprime)&=&\epsilon(yx),\\
r(x^\prime\otimes y^\prime)\bar r(x^\pprime\otimes y^\pprime)&=&\epsilon(xy).
\end{eqnarray}
\item
For all $x,y,z\in H$,
\begin{eqnarray}
x^\prime y^\prime r(x^\pprime\otimes y^\pprime)
&=&r(x^\prime\otimes y^\prime)y^\pprime x^\pprime,\\
\label{eq_univr1}
r((xy)\otimes z)&=&r(y\otimes z^\prime) r(x\otimes z^\pprime),\\
\label{eq_univr2}
r(x\otimes (yz))&=&r(x^\prime\otimes y) r(x^\pprime\otimes z).
\end{eqnarray}
\end{myenumerate}
The WHA $H$ is called \emph{cotriangular} if in addition
\begin{equation}
r(x^\prime\otimes y^\prime) r(y^\pprime\otimes x^\pprime) = \epsilon(xy)
\end{equation}
for all $x,y\in H$.
\end{definition}

\subsection{Comodules of Weak Hopf Algebras}
\label{sect_prelimcomod}

We extend Sweedler's notation to the right $H$-comodules and write
$\beta(v)=v_0\otimes v_1$ for the coaction $\beta\colon V\to V\otimes H$ of
$H$ on some vector space $V$.

\begin{proposition}
Let $H$ be a WBA. Then the category $\sym{M}^H$ of finite-dimensional right
$H$-comodules is a monoidal category
$(\sym{M}^H,\hotimes,H_s,\alpha,\lambda,\rho)$. Here the monoidal unit
object is the source base algebra $H_s$ with the coaction
\begin{equation}
\label{eq_betahs}
\beta_{H_s}\colon H_s\to H_s\otimes H,\qquad
x\mapsto x^\prime\otimes x^\pprime.
\end{equation}
The tensor product $V\hotimes W:=\im P_{V,W}$ of two right $H$-comodules
$V,W\in|\sym{M}^H|$ is the \emph{truncated tensor product}, which is the image
of the $k$-linear idempotent
\begin{equation}
\label{eq_idempotent}
P_{V,W}\colon V\otimes W\to V\otimes W,\quad
v\otimes w\mapsto (v_0\otimes w_0)\epsilon(v_1w_1),
\end{equation}
with the coaction given by
\begin{equation}
\beta_{V\hotimes W}\colon V\hotimes W\to (V\hotimes W)\otimes H,\quad
v\otimes w \mapsto (v_0\otimes w_0)\otimes (v_1w_1).
\end{equation}
The unit constraints of the monoidal category are
\begin{alignat}{3}
\lambda_V &\colon H_s\hotimes V\to V,&&\quad x\otimes v\mapsto v_0\epsilon(xv_1),\\
\rho_V    &\colon V\hotimes H_s\to V,&&\quad v\otimes x\mapsto v_0\epsilon(v_1\epsilon_s(x)),
\end{alignat}
and the associator is inherited from that of $\Vect_k$.
\end{proposition}

The forgetful functor of the category of finite-dimensional comodules of a WBA
is not necessarily strong monoidal as in the case of a bialgebra, but it
satisfies the following more general conditions of a functor with separable
Frobenius structure~\cite{Sz05}.

\begin{definition}
Let $\sym{C}$ and $\sym{C}^\prime$ be monoidal categories. A \emph{functor
with Frobenius structure}
$(F,F_{X,Y},F_0,F^{X,Y},F^0)\colon\sym{C}\to\sym{C}^\prime$ is a functor
$F\colon\sym{C}\to\sym{C}^\prime$ that is lax monoidal as $(F,F_{X,Y},F_0)$
and oplax monoidal as $(F,F^{X,Y},F^0)$ and that satisfies the following
compatibility conditions,
\begin{equation}
\begin{aligned}
\xymatrix{
F(X\otimes Y)\otimes^\prime FZ\ar[rr]^{F_{X\otimes Y,Z}}\ar[dd]_{F^{X,Y}\otimes^\prime\id_{FZ}}&&
F((X\otimes Y)\otimes Z)\ar[rr]^{F\alpha_{X,Y,Z}}&&
F(X\otimes(Y\otimes Z))\ar[dd]^{F^{X,Y\otimes Z}}\\
\\
(FX\otimes^\prime FY)\otimes^\prime FZ\ar[rr]_{\alpha^\prime_{FX,FY,FZ}}&&
FX\otimes^\prime(FY\otimes^\prime FZ)\ar[rr]_{\id_{FX}\otimes^\prime F_{Y,Z}}&&
FX\otimes^\prime F(Y\otimes Z),
}
\end{aligned}
\end{equation}
and
\begin{equation}
\begin{aligned}
\xymatrix{
FX\otimes^\prime F(Y\otimes Z)\ar[rr]^{F_{X,Y\otimes Z}}\ar[dd]_{\id_{FX}\otimes^\prime F^{Y,Z}}&&
F(X\otimes (Y\otimes Z))\ar[rr]^{F\alpha^{-1}_{X,Y,Z}}&&
F((X\otimes Y)\otimes Z)\ar[dd]^{F^{X\otimes Y,Z}}\\
\\
FX\otimes^\prime(FY\otimes^\prime FZ)\ar[rr]_{{\alpha^\prime}^{-1}_{FX,FY,FZ}}&&
(FX\otimes^\prime FY)\otimes^\prime FZ\ar[rr]_{F_{X,Y}\otimes^\prime\id_{FZ}}&&
F(X\otimes Y)\otimes^\prime FZ,
}
\end{aligned}
\end{equation}
for all $X,Y,Z\in|\sym{C}|$. It is called a \emph{functor with separable
Frobenius structure} if in addition
\begin{equation}
F_{X,Y}\circ F^{X,Y} = \id_{F(X\otimes Y)},
\end{equation}
for all $X,Y\in|\sym{C}|$.
\end{definition}

This terminology was chosen because if $\sym{C}^\prime=\Vect_k$, the vector
space $F\one$ forms a Frobenius algebra if $F$ has a Frobenius structure and
an index-one Frobenius algebra if $F$ has a separable Frobenius structure,
respectively. Frobenius algebras over a field are separable if and only if
their Frobenius structure can be chosen to be of index one~\cite{KaSz03}.

\begin{proposition}
\label{prop_forgetful}
Let $(H,\mu,\eta,\Delta,\epsilon)$ be a WBA and
$U\colon\sym{M}^H\to\Vect_k$ be the obvious forgetful functor. Then
$(U,U_{X,Y},U_0,U^{X,Y},U^0)$ is a $k$-linear faithful functor with a
separable Frobenius structure, and it takes values in $\fdVect_k$. The
Frobenius structure is given by
\begin{eqnarray}
U_{X,Y}=\coim P_{X,Y}  \colon UX\otimes UY         &\to& P_{X,Y}(UX\otimes UY),\\
U_0    =\eta           \colon k                    &\to& H_s,\\
U^{X,Y}=\im P_{X,Y}    \colon P_{X,Y}(UX\otimes UY)&\to& UX\otimes UY,\\
U^0    =\epsilon|_{H_s}\colon H_s                  &\to& k.
\end{eqnarray}
Here $P_{X,Y}$ denotes the idempotent of~\eqref{eq_idempotent} with its image
factorization $P_{X,Y}=\im P_{X,Y}\circ\coim P_{X,Y}$. Its image
$P_{X,Y}(UX\otimes UY)=U(X\hotimes Y)$ is the vector space underlying the
truncated tensor product. Finally, $H_s=U\one$ is the vector space underlying
the monoidal unit.
\end{proposition}

\begin{proposition}
\label{prop_autonomous}
Let $H$ be a WHA. Then $\sym{M}^H$ is left-autonomous if the left-dual of
every object $V\in|\sym{M}^H|$ is chosen to be $(V^\ast,\ev_V,\coev_V)$, where
the dual vector space $V^\ast$ is equipped with the coaction
\begin{equation}
\beta_{V^\ast}\colon V^\ast\to V^\ast\otimes H,\qquad
\theta\mapsto (v\mapsto \theta(v_0)\otimes S(v_1)),
\end{equation}
and the evaluation and coevaluation maps are given by
\begin{alignat}{3}
\ev_V   &\colon V^\ast\hotimes V\to H_s,&&\quad \theta\otimes v\to\theta(v_0)\epsilon_s(v_1),\\
\coev_V &\colon H_s\to V\hotimes V^\ast,&&\quad x\mapsto\sum_j ({(v_j)}_0\otimes \theta^j)\epsilon(x{(v_j)}_1).
\end{alignat}
Here we have used the evaluation and coevaluation maps that turn
$V^\ast$ into a left-dual of $V$ in the category $\fdVect_k$:
\begin{alignat}{3}
\ev_V^{(\fdVect_k)}   &\colon V^\ast\otimes V\to k,&&\quad \theta\otimes v\mapsto \theta(v),\\
\coev_V^{(\fdVect_k)} &\colon k\to V\otimes V^\ast,&&\quad 1\mapsto \sum_j v_j\otimes \theta^j.
\end{alignat}
\end{proposition}

Let $V\in\sym{M}^H$ be a finite-dimensional right comodule of a WBA $H$ with
some basis ${\{e_j\}}_j$. Then there are unique elements $c^{(V)}_{\ell j}\in
H$ such that $\beta_V(e_j)=\sum_\ell e_\ell\otimes c^{(V)}_{\ell j}$ for all
$j$. These $c^{(V)}_{\ell j}$ are called the \emph{coefficients} of $V$ and
their linear span the \emph{coefficient coalgebra} $C(V)$. $C(V)$ is a
subcoalgebra of $H$. If $H$ is a WHA, we call the element
$t_V=\sum_jc^{(V)}_{jj}\in H$ the \emph{dual character of} $V$.

\subsection{Tannaka--Kre\v\i n reconstruction}
\label{sect_prelimtk}

\begin{definition}
\label{def_longforget}
Let $\sym{C}$ be an essentially small, finitely split semisimple, $k$-linear,
additive monoidal category such that $k$ is a field and $\Hom(X,Y)$ is
finite-dimensional over $k$ for all $X,Y\in\sym{C}$. By ${\{V_j\}}_{j\in I}$
where $I$ is a finite index set, we denote a set of representatives of the
isomorphism classes of simple objects of $\sym{C}$. Then the \emph{long
canonical functor} is defined as
\begin{eqnarray}
\omega\colon\sym{C}\to\Vect_k,\quad X &\mapsto&\Hom(\hat V,\hat V\otimes X),\\
f &\mapsto& (\id_{\hat V}\otimes f)\circ-,\nn
\end{eqnarray}
where $\hat V$ denotes the object
\begin{equation}
\label{eq_hatv}
\hat V=\bigoplus_{j\in I} V_j.
\end{equation}
\end{definition}

\begin{remark}
\label{rem_endvhat}
The algebra $R:=\End(\hat V)\cong\omega\one\cong k^{|I|}$ has a basis
${(\lambda_j)}_{j\in I}$ of orthogonal idempotents given by
$\lambda_j=\id_{V_j}\in R$. It forms a Frobenius algebra
$(R,\circ,\id_R,\Delta_R,\epsilon_R)$ with comultiplication $\Delta_R\colon
R\to R\otimes R$ and counit $\epsilon_R\colon R\to k$ given by
$\Delta(\lambda_j)=\lambda_j\otimes\lambda_j$ and $\epsilon(\lambda_j)=1$ for
all $j\in I$. The element $\Delta(\id_R)$ is a separability idempotent. Such a
Frobenius algebra is called \emph{index one} or \emph{Frobenius
separable}~\cite{KaSz03}.
\end{remark}

\begin{proposition}
Let $\sym{C}$ be as in Definition~\ref{def_longforget}. Then the long
canonical functor $\omega\colon\sym{C}\to\Vect_k$ is a $k$-linear
faithful functor with a separable Frobenius structure
$(\omega,\omega_{XY},\omega_0,\omega^{XY},\omega^0)$ and takes values
in $\fdVect_k$. The separable Frobenius structure is given by
\begin{eqnarray}
\omega_0\colon
k &\to& \omega\one,\\
1 &\mapsto& \rho_{\hat V}^{-1},\nn\\
\omega_{X,Y}\colon
\omega X\otimes\omega Y &\to& \omega(X\otimes Y),\\
f\otimes g &\mapsto& \alpha_{\hat V,X,Y}\circ(f\otimes\id_Y)\circ g,\nn\\
\omega^0\colon
\omega\one &\to& k,\\
v &\mapsto& \epsilon_R(\rho_{\hat V}\circ v),\nn\\
\omega^{X,Y}\colon
\omega(X\otimes Y) &\to& \omega X\otimes\omega Y,\\
h &\mapsto& \sum_{j,\ell}
\epsilon_R\Bigl(e^\ell_{(Y)}\circ(e^j_{(X)}\otimes\id_Y)
\circ\alpha^{-1}_{\hat V,X,Y}\circ h\Bigr)e_j^{(X)}\otimes e_\ell^{(Y)}.\nn
\end{eqnarray}
Here ${(e_j^{(X)})}_j$ and ${(e^j_{(X)})}_j$ denote a pair of dual bases of
$\omega X=\Hom(\hat V,\hat V\otimes X)$ and $\Hom(\hat V\otimes X,\hat V)$,
respectively, with respect to the non-degenerate bilinear form
\begin{equation}
\label{eq_gx}
g_X\colon\Hom(\hat V\otimes X,\hat V)\otimes\Hom(\hat V,\hat V\otimes X)\to k,\qquad
\theta\otimes v\mapsto\epsilon_R(\theta\circ v).
\end{equation}
\end{proposition}

\begin{remark}
\begin{myenumerate}
\item
It can be shown that the long canonical functor already has a separable
Frobenius structure if each simple object $X\in|\sym{C}|$ has $\End(X)$ a
finite-dimensional separable division algebra over $k$. Such an algebra admits
an index one Frobenius structure~\cite{KaSz03}. For our construction below in
terms of the dimension graph, however, we require the stronger condition that
$\End(X)\cong k$.
\item
Since $\sym{C}$ is semisimple, there is no need to worry about
exactness of $\omega$ at this point. Thanks to the equivalence
$\sym{C}\simeq\sym{M}^H$ in Theorem~\ref{thm_tk} below
(see~\cite{Pf09a}), $\sym{C}$ is abelian and $\omega$ exact.
\end{myenumerate}
\end{remark}

By a generalization of Tannaka--Kre\v\i n reconstruction from strong
monoidal functors to functors with separable Frobenius structure, we
obtain the following characterization of $\sym{C}$ as the category
$\sym{C}\simeq\sym{M}^H$ of finite-dimensional comodules over the
universal coend $H=\coend(\sym{C},\omega)$. The long canonical functor
appears as the forgetful functor $\omega\colon\sym{M}^H\to\Vect_k$.

\begin{theorem}
\label{thm_tk}
Let $\sym{C}$ be as in Definition ~\ref{def_longforget}. Then
$\sym{C}\simeq\sym{M}^H$ are equivalent as $k$-linear, additive monoidal
categories. Here $H=\coend(\sym{C},\omega)$ is a finite-dimensional split
cosemisimple WBA such that $H_s\cong R\cong H_t$. The WBA $H$ is a direct sum
of matrix coalgebras,
\begin{equation}
H = \bigoplus_{j\in I}{(\omega V_j)}^\ast\otimes\omega V_j,
\end{equation}
with operations
\begin{eqnarray}
\mu({[\theta|v]}_X\otimes {[\zeta|w]}_Y)
&=& {[\zeta\circ(\theta\otimes\id_Y)\circ\alpha_{\hat V,X,Y}^{-1}|
\alpha_{\hat V,X,Y}\circ(v\otimes\id_Y)\circ w]}_{X\otimes Y},\\
\eta(1) &=& {[\rho_{\hat V}|\rho_{\hat V}^{-1}]}_\one,\\
\Delta({[\theta|v]}_X)
&=& \sum_j{[\theta|e_j^{(X)}]}_X\otimes {[e^j_{(X)}|v]}_X,\\
\epsilon ({[\theta|v]}_X) &=& \epsilon_R(\theta\circ v).
\end{eqnarray}
Here we write ${[\theta|v]}_X\in{(\omega X)}^\ast\otimes\omega X$ with
$v\in\omega X$, $\theta\in\Hom(\hat V\otimes X,\hat V)\cong{(\omega X)}^\ast$
and simple $X\in|\sym{C}|$ for the homogeneous elements of $H$. The precise
form of the universal coend as a colimit also allows us to use the same
expression for arbitrary objects of $\sym{C}$, but subject to the relations
that $[{\zeta|(\omega f)(v)]}_Y ={[{(\omega f)}^\ast(\zeta)|v]}_X$ for all
$v\in\omega X$, $\zeta\in{(\omega Y)}^\ast$ and for all morphisms $f\colon
X\to Y$ of $\sym{C}$. Recall that $(\omega f)(v)=(\id_{\hat V}\otimes f)\circ
v$ and ${(\omega f)}^\ast(\zeta)=\zeta\circ(\id_{\hat V}\otimes f)$.

If in addition, $\sym{C}$ is left-autonomous, then $H$ forms a WHA with
antipode
\begin{equation}
S({[e^j_{(X)}|e_\ell^{(X)}]}_X)
= {[\tilde e^\ell_{(X^\ast)}|\tilde e_j^{(X^\ast)}]}_{X^\ast}
\end{equation}
where ${(\tilde e_j^{(X^\ast)})}_j$ denotes the basis of $\omega(X^\ast)$
defined by
\begin{equation}
\tilde e_j^{(X^\ast)} = (e^j_{(X)}\otimes\id_{X^\ast})
\circ\alpha^{-1}_{\hat V,X,X^\ast}\circ(\id_{\hat V}\otimes\coev_X)
\circ\rho_{\hat V}^{-1},
\end{equation}
and where ${(\tilde e^j_{(X^\ast)})}_j$ is the basis dual to it with
respect to the bilinear form $g_{X^\ast}$, \cf~\eqref{eq_gx}.
\end{theorem}

\begin{remark}
\begin{myenumerate}
\item
If the monoidal unit $\one$ is simple, the base algebras intersect
trivially, $H_s\cap H_t\cong k$. If $\sym{C}$ is braided, $H$ is
coquasi-triangular and $\sym{C}\simeq\sym{M}^H$ an equivalence of
braided monoidal categories. If $\sym{C}$ is symmetric monoidal, $H$
is cotriangular. Further structure and properties of $\sym{C}$ such
as a pivotal structure, a ribbon structure, or the properties that a
pivotal category $\sym{C}$ be spherical or that a ribbon category
$\sym{C}$ be modular, can be translated into additional structure
and properties of $H=\coend(\sym{C},\omega)$ as
well~\cite{Pf09a,Pf09b}.
\item
Note that if $X\in|\sym{C}|$ is an arbitrary object, then $\omega X$ forms a
right-$H$ comodule with the coaction
\begin{equation}
\beta_{\omega X}\colon\omega X\to\omega X\otimes H,\quad
v\mapsto\sum_je^{(X)}_j\otimes{[e^j_{(X)}|v]}_X.
\end{equation}
Its coefficient coalgebra is given by $C(X)={({(\omega
X)}^\ast\otimes\omega X)}/N_X\subseteq H$ where the subspace
$N_X\subseteq{(\omega X)}^\ast\otimes\omega X$ is generated by the
elements
\begin{equation}
{[\theta|(\omega f)(v)]}_X-{[{(\omega f)}^\ast(\theta)|v]}_X
\end{equation}
for all $v\in\omega X$, $\theta\in{(\omega X)}^\ast$ and
$f\in\End(X)$.
\end{myenumerate}
\end{remark}

\section{A combinatorial cover of the universal coend}
\label{sect_combinatorial}

In order to develop a combinatorial description of a given category
$\sym{C}$ with the properties as in Definition~\ref{def_longforget},
we first construct a WBA $H[\sym{G}]$ in combinatorial terms and a
surjection $\pi\colon H[\sym{G}]\to H$ onto the universal coend
$H=\coend(\sym{C},\omega)$.

\subsection{Weak Bialgebras associated with finite directed graphs}

Let $\sym{G}=(\sym{G}^0,\sym{G}^1)$ be a finite directed graph with a set
$\sym{G}^0$ of vertices and a set $\sym{G}^1\subseteq\sym{G}^0\times\sym{G}^0$
of edges. We use the following notation and terminology. Every edge
$p=(v_0,v_1)\in\sym{G}^1$ has a source and a target vertex, denoted by
$\sigma(p)=v_1\in\sym{G}^0$ and $\tau(p)=v_0\in\sym{G}^0$, respectively. We
also set $\sigma(v)=v=\tau(v)$ for all $v\in\sym{G}^0$. By
\begin{equation}
\sym{G}^m=\{\,(p_1,\ldots,p_m)\in{(\sym{G}^1)}^m\mid\quad
\sigma(p_j)=\tau(p_{j+1})\quad\mbox{for all}\quad 1\leq j\leq m-1\,\},
\end{equation}
we denote the set of paths of length $m$ in $\sym{G}$, $m\in\N$. Finally, for
vertices $v,w\in\sym{G}^0$, the set
\begin{equation}
\sym{G}^m_{wv}=\{\,p\in\sym{G}^m\mid\quad\sigma(p)=v,\quad\tau(p)=w\,\}
\end{equation}
contains all paths of length $m\in\N_0$ from $v$ to $w$.

We write $pq\in\sym{G}^{\ell+m}$ for the concatenation of two paths
$p\in\sym{G}^\ell$ and $q\in\sym{G}^m$ provided that $\sigma(p)=\tau(q)$. The
free $k$-vector space on the set $\sym{G}^m$ is denoted by $k\sym{G}^m$,
$m\in\N_0$.

\begin{proposition}
\label{prop_defhg}
Let $\sym{G}$ be a finite directed graph. Then there is a WBA
$(H[\sym{G}],\mu,\eta,\Delta,\epsilon)$ with the underlying vector space
\begin{equation}
H[\sym{G}] = \coprod_{m\in\N_0}{(k\sym{G}^m)}^\ast\otimes k\sym{G}^m
\end{equation}
and operations
\begin{eqnarray}
\eta(1)
&=& \sum_{j,\ell\in\sym{G}^0}{[j|\ell]}_0,\\
\mu({[p|q]}_m\otimes {[r|s]}_\ell)
&=& \delta_{\sigma(p),\tau(r)}\delta_{\sigma(q),\tau(s)}{[pr|qs]}_{m+\ell},\\
\Delta({[p|q]}_m)
&=& \sum_{r\in\sym{G}^m}{[p|r]}_m\otimes {[r|q]}_m,\\
\epsilon({[p|q]}_m)
&=& \delta_{pq},
\end{eqnarray}
for all $p,q\in\sym{G}^m$, $r,s\in\sym{G}^\ell$, $m,\ell\in\N_0$.  Here,
$\coprod$ is the coproduct in $\Vect_k$, and we have denoted basis vectors of
the homogeneous components ${H[\sym{G}]}_m={(k\sym{G}^m)}^\ast\otimes
k\sym{G}^m$ of $H[\sym{G}]$ by ${[p|q]}_m=p\otimes q\in
{(k\sym{G}^m)}^\ast\otimes k\sym{G}^m$. As usual, we write $\delta_{pq}=1$ if
$p=q$ and $\delta_{pq}=0$ if $p\neq q$ for all $p,q\in\sym{G}^m$, $m\in\N_0$.
\end{proposition}

\begin{proof}
Direct verification.
\end{proof}

\begin{proposition}
\label{prop_hg}
Let $\sym{G}$ be a finite directed graph and $H[\sym{G}]$ as in
Proposition~\ref{prop_defhg}.
\begin{myenumerate}
\item
The source and target counital maps of $H[\sym{G}]$ are given by
\begin{eqnarray}
\epsilon_s({[p|q]}_m)
&=& \delta_{p,q}\,\sum_{j\in\sym{G}^0}{[j|\sigma(p)]}_0,\\
\epsilon_t({[p|q]}_m)
&=& \delta_{p,q}\,\sum_{j\in\sym{G}^0}{[\tau(q)|j]}_0,
\end{eqnarray}
for all $p,q\in\sym{G}^m$, $m\in\N_0$.
\item
$H[\sym{G}]$ is split cosemisimple. Its simple right comodules are the
vector spaces $k\sym{G}^m$, $m\in\N_0$, with the coactions
\begin{equation}
\beta_{k\sym{G}^m}\colon k\sym{G}^m\to k\sym{G}^m\otimes H[\sym{G}],
p\mapsto\sum_{q\in\sym{G}^m} q\otimes{[q|p]}_m.
\end{equation}
\item
The truncated tensor product of $\sym{M}^{H[\sym{G}]}$ is such that
\begin{equation}
k\sym{G}^m\hotimes k\sym{G}^\ell\cong k\sym{G}^{m+\ell}
\end{equation}
for all $m,\ell\in\N_0$.
\item
The unital algebra underlying $H[\sym{G}]$ is graded with
homogeneous components ${H[\sym{G}]}_m$ of degree $m\in\N_0$.
\item
As an associative algebra or as a unital associative algebra, $H[\sym{G}]$
is generated by the set ${H[\sym{G}]}_0\cup{H[\sym{G}]}_1$.
\end{myenumerate}
\end{proposition}

\begin{proof}
Part~1 is established by a direct computation. Part~2 holds because the
homogeneous components ${H[\sym{G}]}_m$ are matrix coalgebras with coefficients
in $k$. For Part~3, we compute the idempotent~\eqref{eq_idempotent} and find
that for all $m,\ell\in\N_0$, $p\in\sym{G}^m$, $q\in\sym{G}^\ell$,
\begin{equation}
P_{k\sym{G}^m,k\sym{G}^\ell}(p\otimes q) = \left\{
\begin{matrix}
p\otimes q & \mbox{if}\quad\sigma(p)=\tau(q),\\
0          & \mbox{otherwise}.
\end{matrix}\right.
\end{equation}
Part~4 holds because multiplication in $H[\sym{G}]$ is zero unless the paths
in both components of $[-|-]$ are composable. Part~5 holds because the length
of paths is additive under concatenation.
\end{proof}

\begin{remark}
\begin{myenumerate}
\item
The algebra underlying $H[\sym{G}]$ is the \emph{path algebra}
$k\Gamma$ of the \emph{quiver} $\Gamma=\sym{G}\times\sym{G}$, up to
identifying ${(k\sym{G}^m)}^\ast=k\sym{G}^m$. We do not use the
terminology \emph{quiver} in the present article because it is not
the category of modules over $k\Gamma$, but rather that of comodules
that is related to our fusion category $\sym{C}$.
\item
The category $\sym{M}^{H[\sym{G}]}$ of finite-dimensional
right-$H[\sym{G}]$ comodules is an essentially small, split
semisimple, $k$-linear, abelian monoidal category whose isomorphism
classes of simple objects are indexed by non-negative integers
$m\in\N_0$. The tensor product is given by $m\otimes\ell\cong
m+\ell$ for all $m,\ell\in\N_0$. The forgetful functor
$U\colon\sym{M}^{H[\sym{G}]}\to\Vect_k$ is such that $Um\cong
k\sym{G}^m$. Conversely, $H[\sym{G}]\cong\coend(\sym{M}^{H[\sym{G}]},U)$.
\end{myenumerate}
\end{remark}

\subsection{The fundamental surjection}

Although the truncation of the tensor product in
Proposition~\ref{prop_hg}(3) is rather elementary, this is the
mechanism that controls the truncation of the tensor product in all
fusion categories. We demonstrate this by constructing a surjection
$H[\sym{G}]\to H$.

\begin{definition}
\label{def_generates}
Let $\sym{C}$ be an essentially small, finitely split semisimple,
$k$-linear, additive left-autonomous monoidal category such that $k$
is a field and $\Hom(X,Y)$ is finite-dimensional over $k$ for all
$X,Y\in|\sym{C}|$. An object $M\in|\sym{C}|$ is said to
\emph{generate} $\sym{C}$ \emph{as a fusion category} if the following
conditions are satisfied.
\begin{myenumerate}
\item
Every simple object $V_j$, $j\in I$, of $\sym{C}$ appears as a direct
summand of $M^{\otimes n}$ for some $n\in\N_0$.
\item
The object $M$ is \emph{multiplicity free}, \ie\ if $M\cong X\oplus Y\oplus
Z$, then $X\not\cong Y$.
\item
The monoidal unit $\one$ and $M$ have pairwise non-isomorphic direct
summands, \ie\ if $\one\cong X\oplus Y$ and $M\cong Z\oplus W$, then
$X\not\cong Z$.
\end{myenumerate}
\end{definition}

\begin{remark}
\begin{myenumerate}
\item
Part~(1) is the usual definition, but~(2) and~(3) can be required in
addition without loss of generality. Note that Part~(3) rules out
the trivial fusion category, but every non-trivial such category
does have a generating object.
\item
The monoidal unit $\one$ is always
multiplicity-free~\cite{EtNi05}. In the present section and in
Section~\ref{sect_schurweyl}, the assumption that $\sym{C}$ be
autonomous can be dropped if one requires instead that the monoidal
unit be multiplicity-free.
\end{myenumerate}
\end{remark}

Given a fusion category $\sym{C}$ with a generating object
$M\in|\sym{C}|$, we now choose a graph $\sym{G}$ in such a way that we
obtain a surjection of WBAs $\pi\colon H[\sym{G}]\to H$. Then the
composability of paths in $\sym{G}$ which controls the multiplication
in $H[\sym{G}]$, also governs the truncated tensor product in
$\sym{C}\cong\sym{M}^H$.

\begin{definition}
\label{def_graph}
Let $\sym{C}$ be an essentially small, finitely split semisimple,
$k$-linear, additive, left-autonomous monoidal category such that $k$
is a field and $\Hom(X,Y)$ is finite-dimensional over $k$ for all
$X,Y\in|\sym{C}|$. Let $M\in|\sym{C}|$ be an object that generates
$\sym{C}$ and let ${\{V_j\}}_{j\in I}$ denote a set of representatives
of the isomorphism classes of simple objects of $\sym{C}$.  The
\emph{dimension graph $\sym{G}$ of $\sym{C}$ with respect to $M$} is
the finite directed graph whose set of vertices is $\sym{G}^0=I$ and
whose set $\sym{G}^1_{\ell j}$ of edges from $j\in\sym{G}^0$ to
$\ell\in\sym{G}^0$ is a basis of $\Hom(V_j,V_\ell\otimes M)$.
\end{definition}

\begin{remark}
If $M\in|\sym{C}|$ is simple and, say, $M\cong V_1$, then the
adjacency matrix of $\sym{G}$ is the fusion matrix $N_1$ with
coefficients ${(N_1)}_{j\ell}=\dim_k\Hom(V_j,V_\ell\otimes V_1)$.
\end{remark}

We denote a basis of $\omega M=\Hom(\hat V,\hat V\otimes M)$ by
${\{e_p^{(M)}\}}_p$ and by ${\{e_{(M)}^q\}}_q$ its dual basis with respect to
the bilinear form $g_M$ of~\eqref{eq_gx}. We also choose the basis
${\{e^{(\one)}_j\}}_j$, $e^{(\one)}_j:=\rho_{\hat V}^{-1}\circ\lambda_j$ of
$\omega\one=\Hom(\hat V,\hat V\otimes\one)$ whose dual basis with respect to
$g_\one$ is given by ${\{e_{(\one)}^\ell\}}_\ell$ with
$e_{(\one)}^\ell:=\lambda_\ell\circ\rho_{\hat V}$. Observe that
$k\sym{G}^1=\omega M$ and $k\sym{G}^0\cong\omega\one$ which we identify in the
following.

\begin{theorem}
\label{thm_surjection}
Let $\sym{C}$ be as in Definition~\ref{def_graph}, $M\in|\sym{C}|$ be an
object that generates $\sym{C}$ and $\sym{G}$ be the dimension graph of
$\sym{C}$ with respect to $M$. We denote by $H=\coend(\sym{C},\omega)$ the
universal coend with respect to the long canonical functor,
\cf~Theorem~\ref{thm_tk}. Then there is a surjection of WBAs $\pi\colon
H[\sym{G}]\to H$ as follows.
\begin{myenumerate}
\item
$\pi({[j|\ell]}_0)={[e^j_{(\one)}|e^{(\one)}_\ell]}_\one$ for all
$j,\ell\in\sym{G}^0$.
\item
$\pi({[p|q]}_1) = {[e^p_{(M)}|e_q^{(M)}]}_M$ for all $p,q\in\sym{G}^1$.
\item
$\pi$ pushes forward the right $H[\sym{G}]$-comodule $k\sym{G}^1$ to the
right $H$-comodule $\omega M$, \ie\
$(\id_{k\sym{G}^1}\otimes\pi)\circ\beta_{k\sym{G}^1}=\beta_{\omega M}$.
\item
$\pi$ also pushes forward the right $H[\sym{G}]$-comodule $k\sym{G}^0$ to
the right $H$-comodule $\omega\one$, \ie\
$(\id_{k\sym{G}^0}\otimes\pi)\circ\beta_{k\sym{G}^0}=\beta_{\omega\one}$.
\end{myenumerate}
\end{theorem}

\begin{proof}
As an associative algebra, $H[\sym{G}]$ is generated by
${H[\sym{G}]}_0\cup{H[\sym{G}]}_1$, \ie\ Parts~(1) and~(2) fix the value
of $\pi$ on a set of generators of $H[\sym{G}]$.

Given that $\sym{G}$ is the dimension graph, (1) and~(2) are compatible with
the multiplication of $H[\sym{G}]$. This can be seen by computing all products
of degree-$0$ and degree-$1$ terms of $H[\sym{G}]$ and their images under
$\pi$. Therefore, Parts~(1) and~(2) define a unique linear map $\pi\colon
H[\sym{G}]\to H$ which forms a homomorphism of associative algebras. This map
$\pi$ is also compatible with the units as can be seen by inspection.

In order to see that $\pi$ respects the comultiplication, we show in a direct
computation that $\Delta(\pi(a))=(\pi\otimes\pi)(\Delta(a))$ for all
generators $a\in{H[\sym{G}]}_0\cup{H[\sym{G}]}_1$. Then, by induction, if this
claim holds for some $a,b\in H[\sym{G}]$, we also have
\begin{eqnarray}
\Delta(\pi(ab))
&=& \Delta(\pi(a)\pi(b))
= ({\pi(a)}^\prime{\pi(b)}^\prime)\otimes({\pi(a)}^\pprime{\pi(b)}^\pprime)\nn\\
&=& (\pi(a^\prime)\pi(b^\prime))\otimes(\pi(a^\pprime)\pi(b^\pprime))
= \pi(a^\prime b^\prime)\otimes\pi(a^\pprime b^\pprime)\nn\\
&=& (\pi\otimes\pi)(\Delta(ab)),
\end{eqnarray}
because $\pi$ respects the multiplication; because $H$ is a WBA; because of
the assumption; because $\pi$ respects the multiplication; and because
$H[\sym{G}]$ is a WBA.

The map $\pi$ also respects the counit. On generators
$a\in{H[\sym{G}]}_0\cup{H[\sym{G}]}_1$, we see by inspection that
$\epsilon(\pi(a))=\epsilon(a)$. Then, by induction, if this claim holds for
some $a,b\in H[\sym{G}]$, we find that
\begin{eqnarray}
\label{eq_pirespectsepsilon}
\epsilon(\pi(ab))
&=& \epsilon(\pi(a)\pi(b))
= \epsilon(\pi(a)1\pi(b))
= \epsilon(\pi(a)1^\prime)\epsilon(1^\pprime\pi(b))\nn\\
&=& \epsilon(\pi(a){\pi(1)}^\prime)\epsilon({\pi(1)}^\pprime\pi(b))
= \epsilon(\pi(a)\pi(1^\prime))\epsilon(\pi(1^\pprime)\pi(b))\nn\\
&=& \epsilon(\pi(a1^\prime))\epsilon(\pi(1^\pprime b))
= \epsilon(ab),
\end{eqnarray}
because $\pi$ respects the multiplication; $H$ is a WBA; $\pi$ respects the
unit; $\pi$ respects the comultiplication, and $\pi$ respects the
multiplication. The last equality of~\eqref{eq_pirespectsepsilon} is shown by
a direct computation for generic $a={[p|q]}_m\in {H[\sym{G}]}$,
$b={[r|s]}_\ell\in {H[\sym{G}]}$.

At this point, we know that $\pi$ is a homomorphism of WBAs. It is surjective
because $\sym{C}$ is generated by $M$ and already the image
$\pi({H[\sym{G}]}_1)$ exhausts the coefficient coalgebra $C(M)={(\omega
M)}^\ast\otimes\omega M$.

Parts~(3) and~(4) can finally be seen in a direct computation.
\end{proof}

\begin{remark}
If both the monoidal unit $\one$ and the generating object $M$ are
simple, then the restriction
$\pi|_{{H[\sym{G}]}_0\oplus{H[\sym{G}]}_1}\colon
{H[\sym{G}]}_0\oplus{H[\sym{G}]}_1\to H$ is injective.
\end{remark}

\section{Schur--Weyl dual description at fixed powers of $M$}
\label{sect_schurweyl}

In this section, we define a quotient of the WBA $H[\sym{G}]$ in such
a way that the tensor powers of the generating object $M$ have the
desired endomorphism algebras.

\subsection{Implementing the endomorphisms of the tensor powers of $M$}

\begin{definition}
\label{def_schurweylsystem}
Let $\sym{C}$ and $M$ be as in Theorem~\ref{thm_surjection}. An
\emph{endomorphism system for $\sym{C}$ with respect to $M$} is a
sequence $\sym{E}={(E^{(n)})}_{n\in\N_0}$ of sets
$E^{(n)}\subseteq\End(M^{\otimes n})$ of endomorphisms such that
\begin{myenumerate}
\item
$\End(\one)$ as an associative algebra is generated by
$E^{(0)}\cup\{\id_\one\}$.
\item
For all $n\in\N$, $\End(M^{\otimes n})$ is generated by
\begin{equation}
(E^{(n-1)}\otimes\id_M)\cup E^{(n)},
\end{equation}
where we have abbreviated
\begin{equation}
E^{(n-1)}\otimes\id_M = \{\,f^{(n-1)}\otimes\id_M\mid\quad f^{(n-1)}\in E^{(n-1)}\,\}.
\end{equation}
\end{myenumerate}
\end{definition}

The situation is particularly easy if both $\one$ and $M$ are simple and if
$\sym{C}$ is braided with braiding $\psi_{X,Y}\colon X\otimes Y\to Y\otimes X$
such that the endomorphism algebras are already generated by the braiding and
inverse braiding of adjacent tensor factors. For $n\geq 2$, we denote by $B_n$
the associative unital algebra generated by $\{\,\psi^\pm_j\mid\quad 1\leq
j\leq n-1\,\}$ with
\begin{equation}
\psi^\pm_j:=\id_{M^{\hotimes(j-1)}}\otimes\psi^\pm_{M,M}
\otimes\id_{M^{\hotimes(n-j-1)}}.
\end{equation}

\begin{definition}
\label{def_schurweyl}
Let $\sym{C}$ and $M$ be as in Theorem~\ref{thm_surjection}. Then
$\sym{C}$ is said to satisfy the \emph{strong Schur--Weyl property} if
both $\one$ and $M$ are simple and if $\sym{C}$ is braided such that
$\End(M^{\otimes n})=B_n$ for all $n\geq 2$.
\end{definition}

\begin{example}
\label{ex_typea}
Let $\sym{C}$ and $M$ be as in Theorem~\ref{thm_surjection} and assume
that $\sym{C}$ satisfies the strong Schur--Weyl property. Then an
endomorphism system for $\sym{C}$ with respect to $M$ is given by
$E^{(0)}=\emptyset$, $E^{(1)}=\emptyset$,
\begin{eqnarray}
E^{(2)} &=& \{\,\psi_{M,M},\,\psi^{-1}_{M,M}\,\},\\
E^{(m)} &=& \{\,\id_{M^{\hotimes(m-2)}}\otimes\psi_{M,M},\,
\id_{M^{\hotimes(m-2)}}\otimes\psi^{-1}_{M,M}\,\},
\end{eqnarray}
for all $m\geq 3$.
\end{example}

Given an endomorphism system $\sym{E}={(E^{(n)})}_{n\in\N_0}$ for
$\sym{C}$ with respect to $M$, we can express the endomorphisms
$\omega f^{(n)}\colon{(\omega M)}^{\hotimes n}\to{(\omega
M)}^{\hotimes n}$, $f^{(n)}\in E^{(n)}$, $n\in\N$, as
\begin{equation}
\label{eq_coefficients}
(\omega f^{(n)})(e^{(M)}_{p_1}\otimes\cdots\otimes e^{(M)}_{p_n}) =
\sum_{r_1,\ldots,r_n\in\sym{G}^1}e^{(M)}_{r_1}\otimes\cdots\otimes e^{(M)}_{r_n}\,
f^{(n)}_{r_1\cdots r_n;p_1\cdots p_n},
\end{equation}
with coefficients $f^{(n)}_{r_1\cdots r_n;p_1\cdots p_n}\in k$. By analogy,
for $n=0$ and $\omega f^{(0)}\colon\omega\one\to\omega\one$, $f^{(0)}\in
E^{(0)}$, this is replaced by
\begin{equation}
(\omega f^{(0)})(e^{(\one)}_j) = \sum_{\ell\in\sym{G}^0}e^{(\one)}_\ell\,f^{(0)}_{\ell;j},
\end{equation}
with coefficients $f^{(0)}_{\ell;j}\in k$.

\begin{remark}
\label{rem_coefficients}
\begin{myenumerate}
\item
Recall that $\End(M^{\otimes n})\cong\End({(\omega M)}^{\hotimes n})$ in view
of the equivalence $\sym{C}\simeq\sym{M}^H$ of $k$-linear monoidal categories.
\item
Because of the form of the long canonical functor, $(\omega
f)(v)=(\id_{\hat V}\otimes f)\circ v$ for $v\in\omega X$, $f\colon
X\to Y$, $X,Y\in|\sym{C}|$, and so the coefficients
$f^{(n)}_{r_1\cdots r_n;p_1\cdots p_n}$ in~\eqref{eq_coefficients}
are zero unless $\tau(r_1)=\tau(p_1)$ and $\sigma(r_n)=\sigma(p_n)$.
\end{myenumerate}
\end{remark}

\begin{definition}
\label{def_schurweyladapted}
Let $\sym{C}$, $M$ and $\sym{G}$ be as in Theorem~\ref{thm_surjection}
and $\sym{E}={(E^{(n)})}_{n\in\N_0}$ be an endomorphism system for
$\sym{C}$ with respect to $M$. The \emph{endomorphism adapted} WBA is
the quotient $H[\sym{G},\sym{E}]:=H[\sym{G}]/I_\sym{E}$ where
$I_\sym{E}$ is the two-sided ideal generated by the relations
\begin{equation}
\label{eq_schurweylrelation}
\sum_{p_1,\ldots,p_n\in\sym{G}^1}{[r_1|p_1]}_1\cdot\cdots\cdot{[r_n|p_n]}_1\,
f^{(n)}_{p_1\cdots p_n;q_1\cdots q_n}
-\sum_{p_1,\ldots,p_n\in\sym{G}^1}f^{(n)}_{r_1\cdots r_n;p_1\cdots p_n}\,
{[p_1|q_1]}_1\cdot\cdots\cdot{[p_n|q_n]}_1
\end{equation}
for all $r_j,q_j\in\sym{G}^1$, $j\in\{1,\ldots,n\}$, $f^{(n)}\in E^{(n)}$ and
$n\in\N_0$.
\end{definition}

Note that all relations in the quotient $H[\sym{G}]/I_\sym{E}$ are
homogeneous, and so $H[\sym{G},\sym{E}]$ is graded,
\cf~Proposition~\ref{prop_hg}(4).

\begin{proposition}
\label{prop_schurweyladapted}
The endomorphism adapted WBA $H[\sym{G},\sym{E}]$ of
Definition~\ref{def_schurweyladapted} forms a WBA.
\end{proposition}

\begin{proof}
We have to show that $I_\sym{E}$ is also a two-sided coideal, \ie\ it
satisfies $\Delta(I_\sym{E})\subseteq I_\sym{E}\otimes H + H\otimes I_\sym{E}$
and $I_\sym{E}\subseteq\ker\epsilon$. This is established in a direct
computation. Note that the relations that generate the ideal $I_\sym{E}$, can
be rewritten as
\begin{eqnarray}
&&\sum_{\ontop{p_1,\ldots,p_n\in\sym{G}^1\colon}{\sigma(p_j)=\tau(p_{j+1})}}
{[r_1\cdots r_n|p_1\cdots p_n]}_nf^{(n)}_{p_1\cdots p_n;q_1\cdots q_n}\nn\\
&-&\sum_{\ontop{p_1,\ldots,p_n\in\sym{G}^1\colon}{\sigma(p_j)=\tau(p_{j+1})}}
f^{(n)}_{r_1\cdots r_n;p_1\cdots p_n}{[p_1\cdots p_n|q_1\cdots q_n]}_n.
\end{eqnarray}
\end{proof}

\begin{proposition}
\label{prop_factorendo}
Under the assumptions of Definition~\ref{def_schurweyladapted}, the surjection
$\pi\colon H[\sym{G}]\to H$ of Theorem~\ref{thm_surjection} factors through
the canonical projection $H[\sym{G}]\to H[\sym{G},\sym{E}]$, giving rise to
another surjection of WBAs $\bar\pi\colon H[\sym{G},\sym{E}]\to H$. This map
$\bar\pi$ also satisfies the properties~(1) to~(4) of
Theorem~\ref{thm_surjection} and in addition
\begin{myitemize}
\item[5.]
The restriction $\bar\pi|_{{H[\sym{G},\sym{E}]}_0\oplus{H[\sym{G},\sym{E}]}_1}\colon
{H[\sym{G},\sym{E}]}_0\oplus{H[\sym{G},\sym{E}]}_1\to H$ is injective.
\end{myitemize}
\end{proposition}

\begin{proof}
In any quotient of $H[\sym{G}]$, the relation~\eqref{eq_schurweylrelation} holds
for a particular $f^{(n)}\in E^{(n)}$, $n\in\N_0$, if and only if the linear
map
\begin{equation}
{(k\sym{G}^1)}^{\hotimes n}\to {(k\sym{G}^1)}^{\hotimes n},\qquad
p_1\otimes\cdots\otimes p_n\mapsto
\sum_{r_1,\ldots,r_n\in\sym{G}^1}r_1\otimes\cdots\otimes r_n\,
f^{(n)}_{r_1\cdots r_n;p_1\cdots p_n}
\end{equation}
forms a morphism of right comodules. This is established in a direct
calculation. The claim holds because $\pi$ pushes forward $k\sym{G}^0$ to
$\omega\one$ and $k\sym{G}^1$ to $\omega M$, and because $\pi$ is a morphism
of WBAs and therefore preserves tensor products. For Property~(5), we recall
that $M$ satisfies all conditions of Definition~\ref{def_generates} and that
the direct summands of $\one$ are pairwise non-isomorphic~\cite{EtNi05}.
\end{proof}

The special case of the endomorphism system of Example~\ref{ex_typea} is
particularly interesting because in this case, the ideal $I_\sym{E}$ is
generated by quadratic relations only.

\begin{definition}
\label{def_weakfrt}
Let $\sym{C}$, $M$ and $\sym{G}$ be as in Theorem~\ref{thm_surjection}
and the endomorphism system $\sym{E}={(E^{(n)})}_{n\in\Z_{\geq 0}}$ be
as in Example~\ref{ex_typea}. We denote the coefficients of the
braiding under the long canonical functor by $R_{r_1r_2;p_1p_2}$, \ie
\begin{equation}
\omega(\psi_{M,M})(p_1\otimes p_2)
= \sum_{r_1,r_2\in\sym{G}^1} r_1\otimes r_2\,R_{r_1r_2;p_1p_2},
\end{equation}
for all $p_1,p_2\in\sym{G}^1$. The \emph{Weak
Faddeev--Reshetikhin--Takhtajan} (FRT) \emph{Bialgebra} is the quotient
$H[\sym{G},R]:=H[\sym{G}]/I_R$ where $I_R$ is the two-sided ideal generated by
the relations
\begin{equation}
\label{eq_frtrelation}
\sum_{p_1,p_2\in\sym{G}^1}{[r_1|p_1]}_1\cdot{[r_2|p_2]}_1\,R_{p_1p_2;q_1q_2}
- \sum_{p_1,p_2\in\sym{G}^1}R_{r_1r_2;p_1p_2}\,{[p_1|q_1]}_1\cdot{[p_2|q_2]}_1,
\end{equation}
for all $r_1,r_2,q_1,q_2\in\sym{G}^1$.
\end{definition}

Note that the relations can again be written in a slightly different fashion:
\begin{equation}
\sum_{\ontop{p_1,p_2\in\sym{G}^1\colon}{\sigma(p_1)=\tau(p_2)}}
{[r_1r_2|p_1p_2]}_2\,R_{p_1p_2;q_1q_2}
- \sum_{\ontop{p_1,p_2\in\sym{G}^1\colon}{\sigma(p_1)=\tau(p_2)}}
R_{r_1r_2;p_1p_2}\,{[p_1p_2|q_1q_2]}_2.
\end{equation}
In complete analogy to Proposition~\ref{prop_schurweyladapted}, the
Weak FRT Bialgebra $H[\sym{G},R]$ forms a WBA. We call it the Weak FRT
Bialgebra because our quotient generalizes the construction of
Faddeev--Reshetikhin--Takhtajan~\cite{ReTa90} to finite-dimensional
split cosemisimple WBAs. This is the situation in which Hayashi
describes WHAs whose categories of finite-dimensional comodules have
the same fusion rules as the modular categories associated with
$U_q(\mathfrak{sl}_N)$ at suitable roots of unity~\cite{Ha99a}. Since
$R$ consists of the coefficients of the braiding, $R$ satisfies a
generalization of the Quantum Yang--Baxter equation to truncated
tensor products. Such an $R$-matrix is known in the physics literature
as the Boltzmann weight of a star-triangular face model. $R$-matrices
of this type were studied, for example, in~\cite{JiMi88}.

\begin{example}
If $\sym{C}$ and its endomorphism system $\sym{E}={(E^{(n)})}_{n\in\N_0}$ are as
in Example~\ref{ex_typea}, the endomorphism adapted WBA coincides with the
Weak FRT Bialgebra.
\end{example}

\begin{proof}
We have to show that the quadratic relations that define $I_R$
in~\eqref{eq_frtrelation} already generate the entire two-sided ideal
$I_\sym{E}$ of Definition~\ref{def_schurweyladapted}. Recall that
$E^{(0)}=\emptyset=E^{(1)}$ in this example. First, since $R$ is the
braiding, each relation~\eqref{eq_frtrelation} for $R$ implies another
relation of the same form with $R^{-1}$ rather than $R$. Just multiply
the relation~\eqref{eq_frtrelation} with $R^{-1}$ from the right and
from the left. Second, in order to show that all
elements~\eqref{eq_schurweylrelation} of degree greater than two are
already contained in the two-sided ideal generated by $I_R$, it is
sufficient to verify that all endomorphisms of the form $\omega
f=\id_{{(\omega M)}^{\hotimes(n-2)}}\otimes\omega(\sigma^\pm_{M,M})$,
$f\in E^{(n)}$, $n\geq 3$, are already implemented by the quotient
modulo $I_R$. This is done in a direct computation.
\end{proof}

Even in situations in which $\sym{C}$ is braided and the braiding and
its inverse do not generate all endomorphisms of the tensor powers of
$\omega M$, the Weak FRT Bialgebra is worth studying in more
detail. Firstly, the Weak FRT Bialgebra is equipped with a
coquasi-triangular structure and, secondly, the endomorphism adapted
WBA is a quotient of the Weak FRT Bialgebra.

\begin{proposition}
\label{prop_coquasi}
Let $\sym{C}$, $M$ and $\sym{G}$ be as in Theorem~\ref{thm_surjection}
and assume in addition that $\sym{C}$ is braided. Then the WBA
$H[\sym{G},R]$ of Definition~\ref{def_weakfrt} is coquasi-triangular
with universal $r$-form $r\colon H[\sym{G},R]\otimes H[\sym{G},R]\to
k$ given by
\begin{eqnarray}
r({[u|v]}_0\otimes{[w|x]}_0) &=& \delta_{u,v}\delta_{v,x}\delta_{u,w},\\
r({[u|v]}_0\otimes{[p|q]}_1)
&=& \delta_{u,\tau(p)}\delta_{v,\sigma(p)}\delta_{p,q},\\
r({[p|q]}_0\otimes{[u|v]}_0)
&=& \delta_{\sigma(p),u}\delta_{\tau(p),v}\delta_{p,q},\\
r({[p|q]}_1\otimes{[r|s]}_1)
&=& \left\{\begin{matrix}
R_{pr,sq}&\mbox{if}\quad\begin{matrix}\tau(p)=\tau(s),\sigma(r)=\sigma(q),\\\sigma(s)=\tau(q),\sigma(p)=\tau(r),\end{matrix}\\
0&\mbox{otherwise}
\end{matrix}\right.
\end{eqnarray}
and further, inductively, by
\begin{equation}
r({[p|q]}_m\otimes{[r|s]}_{\ell+1})
= \sum_{t\in\sym{G}^m} r({[p|t]}_m\otimes{[r_1|s_1]}_\ell)
r({[t|q]}_m\otimes{[r_2|s_2]}_1),
\end{equation}
where $r=r_1r_2$, $s=s_1s_2$ with $r_1,s_1\in\sym{G}^\ell$ and
$r_2,s_2\in\sym{G}^1$; and by
\begin{equation}
r({[p|q]}_{m+1}\otimes{[r|s]}_\ell)
= \sum_{t\in\sym{G}^\ell} r({[p_2|q_2]}_1\otimes{[r|t]}_\ell)
r({[p_1|q_1]}_m\otimes{[t|s]}_\ell),
\end{equation}
where $p=p_1p_2$ and $q=q_1q_2$ with $p_1,q_1\in\sym{G}^m$ and
$p_2,q_2\in\sym{G}^1$. The map $\pi_R=\bar\pi\circ\pi^\prime\colon
H[\sym{G},\sym{R}]\to H$ is a surjective homomorphism of coquasi-triangular
WBAs.
\end{proposition}

\begin{proof}
Direct computation.
\end{proof}

\begin{corollary}
\label{cor_coquasi}
Let $\sym{C}$, $M$ and $\sym{G}$ be as in Theorem~\ref{thm_surjection}
and assume in addition that $\sym{C}$ is braided. Then there is a
surjection of coquasi-triangular WBAs $\pi^\prime\colon
H[\sym{G},R]\to H[\sym{G},\sym{E}]$ and the \emph{Boltzmann weight}
\begin{equation}
R\colon k\sym{G}^1\hotimes k\sym{G}^1\to k\sym{G}^1\hotimes k\sym{G}^1,\qquad
p_1\otimes p_2\mapsto\sum_{r_1,r_2\in\sym{G}^1}r_1\otimes r_2\,R_{r_1r_2;p_1p_2},
\end{equation}
is \emph{star-triangular}, \ie\
\begin{equation}
R_1\circ R_2\circ R_1 = R_2\circ R_1\circ R_2,
\end{equation}
where $R_1=R\otimes\id_{\omega M}$ and $R_2=\id_{\omega M}\otimes R$.
\end{corollary}

\begin{proof}
Since for all $n\geq 2$, $B_n\leq\End(M^{\otimes n})$ always forms a
subalgebra, we have $I_R\subseteq I_\sym{E}$. A compatible
coquasi-triangular structure on $H[\sym{G},\sym{E}]$ always exists
because the braiding is, of course, a morphism. The Boltzmann weight
is star triangular because $R$ coincides with the coefficients of the
braiding $\omega(\psi_{M,M})\colon \omega M\hotimes\omega M\to\omega
M\hotimes\omega M$.
\end{proof}

\subsection{Comparing the categories of comodules}

In this section, we compare the categories $\sym{M}^{H[\sym{G},\sym{E}]}$ and
$\sym{M}^H$ of finite-dimensional comodules of the endomorphism adapted WBA
$H[\sym{G},\sym{E}]$ and of the universal coend $H=\coend(\sym{C},\omega)$.

\begin{proposition}
\label{prop_schurweyl}
Let $\sym{C}$, $M$ and $\sym{G}$ be as in Theorem~\ref{thm_surjection}
and $\sym{E}={(E^{(n)})}_{n\in\N_0}$ be an endomorphism system for
$\sym{C}$ with respect to $M$.
\begin{myenumerate}
\item
The vector space ${(\omega M)}^{\hotimes n}$, $n\in\N_0$, is both an
$H[\sym{G},\sym{E}]$- and an $H$-comodule. The homomorphism of WBAs
$\bar\pi\colon H[\sym{G},\sym{E}]\to H$ pushes forward the
$H[\sym{G},\sym{E}$]-comodule structure to the $H$-comodule structure, \ie
\begin{equation}
(\id_{{(\omega M)}^{\hotimes n}}\otimes\bar\pi)
\circ\beta^{(H[\sym{G},\sym{E}])}_{{(\omega M)}^{\hotimes n}}
= \beta^{(H)}_{{(\omega M)}^{\hotimes n}}.
\end{equation}
\item
A linear map $f\colon{(\omega M)}^{\hotimes n}\to{(\omega M)}^{\hotimes
n}$, $n\in\N_0$, is $H$-colinear if and only if it is
$H[\sym{G},\sym{E}]$-colinear, \ie\
\begin{equation}
\End_{\sym{M}^{H[\sym{G},\sym{E}]}}({(\omega M)}^{\hotimes n})
= \End_{\sym{M}^H}({(\omega M)}^{\hotimes n}).
\end{equation}
\item
A linear subspace $V\subseteq{(\omega M)}^{\hotimes n}$, $n\in\N_0$, is an
$H[\sym{G},\sym{E}]$-subcomodule if and only if it is an $H$-subcomodule.
\item
Let $V\subseteq{(\omega M)}^{\hotimes n}$ and $W\subseteq{(\omega
M)}^{\hotimes m}$, $n,m\in\N_0$, be right
$H[\sym{G},\sym{E}]$-subcomodules. If $n\neq m$, then $V\not\cong W$
are not isomorphic as $H[\sym{G},\sym{E}]$-comodules. If $n=m$, then
$V$ and $W$ are isomorphic as $H[\sym{G},\sym{E}]$-comodules if and
only if they are isomorphic as $H$-comodules.
\item
$H[\sym{G},\sym{E}]$ is split cosemisimple. The isomorphism classes
of its simple comodules can be represented by some
$V^{(m)}_j\in|\sym{M}^{H[\sym{G},\sym{E}]}|$ where $m\in\N_0$ and
$j\in\{1,\ldots,\ell_m\}$, $\ell_m\in\N$. Here, $j$ labels the
isomorphism classes of simple $H$-comodules in the complete
decomposition of ${(\omega M)}^{\hotimes m}$ as an
$H$-comodule. Furthermore, $\bar\pi$ pushes forward each $V^{(m)}_j$
to that particular isomorphism type of $H$-comodules.
\end{myenumerate}
\end{proposition}

\begin{proof}
\begin{enumerate}
\item
Theorem~\ref{thm_surjection}, Parts~(3) and~(4) and the fact that $\bar\pi$
is a homomorphism of algebras.
\item
If $f$ is $H[\sym{G},\sym{E}]$-colinear, then it is also
$H$-colinear, using the homomorphism of coalgebras $\bar\pi\colon
H[\sym{G},\sym{E}]\to H$. Conversely, if $f$ is $H$-colinear, then
$f$ is contained in the $k$-linear span of the set of all finite
products of elements of the form $\omega f^{(m)}\otimes\id_{{(\omega
M)}^{\hotimes(n-m)}}$, $0\leq m\leq n$, $f^{(m)}\in
E^{(m)}$. Each endomorphism of the spanning set is right
$H[\sym{G},\sym{E}]$-colinear by Proposition~\ref{prop_factorendo},
and so $f$ is $H[\sym{G},\sym{E}]$-colinear as well.
\item
Let $V$ be an $H[\sym{G},\sym{E}]$-subcomodule, \ie\ we have
$\beta^{(H[\sym{G},\sym{E}])}_{{(\omega M)}^{\hotimes n}}(V)\subseteq V\otimes
H[\sym{G},\sym{E}]$. Applying $(\id_{{(\omega M)}^{\hotimes
n}}\otimes\bar\pi)$ shows that $\beta^{(H)}_{{(\omega M)}^{\hotimes
n}}(V)\subseteq V\otimes H$. Conversely, let $V$ be an
$H$-subcomodule. Since $H$ is cosemisimple, ${(\omega M)}^{\hotimes n}=V\oplus
W$ with some subcomodule $W$. The linear map $f\colon{(\omega M)}^{\hotimes
n}\to{(\omega M)}^{\hotimes n}$ with $f|_V=\id_V$ and $f|_W=0$ forms a
morphism of $H$-comodules. By Part~2, $f$ is $H[\sym{G},\sym{E}]$-colinear
and therefore its image $V$ an $H[\sym{G},\sym{E}]$-subcomodule.
\item
If $n\neq m$, the coefficient coalgebras
$C(V)\subseteq{H[\sym{G},\sym{E}]}_n$ and
$C(W)\subseteq{H[\sym{G},\sym{E}]}_m$ are in different degree, and
so $C(V)\cap C(W)=\{0\}$ which implies the claim. In the case in
which $n=m$, we use Part~(2).
\item
In order to show that $H[\sym{G},\sym{E}]$ is split cosemisimple, we
show that it is a coproduct (direct sum) of matrix coalgebras.

First, since the ideal $I_\sym{E}$ in
Definition~\ref{def_schurweyladapted} is generated by homogeneous
elements, the canonical projection $p\colon H[\sym{G}]\to
H[\sym{G},\sym{E}]=H[\sym{G}]/I_\sym{E}$ preserves the grading of
the algebra. Therefore,
\begin{equation}
\label{eq_coproduct}
H[\sym{G},\sym{E}] = \coprod_{m\in\N_0}{H[\sym{G},\sym{E}]}_m.
\end{equation}
Here, the homogeneous components
${H[\sym{G},\sym{E}]}_m=p({(k\sym{G}^m)}^\ast\otimes k\sym{G}^m)$ form
the coefficient coalgebras ${H[\sym{G},\sym{E}]}_m=C({(\omega
M)}^{\hotimes m})$. Recall that $p$ pushes forward the
$H[\sym{G}]$-comodule $k\sym{G}^m$ to the
$H[\sym{G},\sym{E}]$-comodule ${(\omega M)}^{\hotimes m}$
(Theorem~\ref{thm_surjection}(3) and Part~1).

Since $H$ is split cosemisimple, ${(\omega M)}^{\hotimes m}$ as an
$H$-comodule decomposes into
\begin{equation}
\label{eq_decompose}
{(\omega M)}^{\hotimes m}\cong
\underbrace{V^{(m)}_1\oplus\cdots\oplus V^{(m)}_1}_{p_1}
\oplus\cdots\oplus
\underbrace{V^{(m)}_{\ell_m}\oplus\cdots\oplus V^{(m)}_{\ell_m}}_{p_{\ell_m}},\qquad \ell_m\in\N.
\end{equation}
Here, the $V^{(m)}_j$, $1\leq j\leq\ell_m$ are pairwise non-isomorphic
$H$-comodules, and the $p_j\in\N$ denote their
multiplicities. Furthermore, $\End_{\sym{M}^H}(V^{(m)}_j)\cong k$ for
all $j$.

By Part~3, each instance of a $V^{(m)}_j\subseteq {(\omega
M)}^{\hotimes m}$ forms an $H[\sym{G},\sym{E}]$-subcomodule, and by
Part~2, $\End_{\sym{M}^{H[\sym{G},\sym{E}]}}(V^{(m)}_j)\cong k$, and
so~\eqref{eq_decompose} is a complete decomposition of
$H[\sym{G},\sym{E}]$-comodules as well.

If we write $V={(\omega M)}^{\hotimes m}$ and
$E=\End_{\sym{M}^{H[\sym{G},\sym{E}]}}(V)$, the construction of
$H[\sym{G},\sym{E}]$ in Definitions~\ref{def_schurweylsystem}
and~\ref{def_schurweyladapted} shows that
\begin{equation}
{H[\sym{G},\sym{E}]}_m = V^\ast\otimes_E V.
\end{equation}
Thanks to~\eqref{eq_decompose}, $E$ is known to be the product of matrix
algebras
\begin{equation}
E=\bigoplus_{j=1}^{\ell_m} k^{p_j\times p_j},
\end{equation}
and a direct computation shows that
\begin{equation}
{H[\sym{G},\sym{E}]}_m = V^\ast\otimes_E V = \bigoplus_{j=1}^{\ell_m}{(V^{(m)}_j)}^\ast\otimes V^{(m)}_j.
\end{equation}
Combining this with~\eqref{eq_coproduct} proves the claim.
\end{enumerate}
\end{proof}

\section{Comparing different powers of $M$}
\label{sect_monoid}

In order to fully understand the relationship between the categories
of comodules of $H[\sym{G},\sym{E}]$ and of $H$, we determine the
preimage under $\hat\pi$ (Proposition~\ref{prop_schurweyl}(5)) of the
simple comodules of $H$, \ie\ all those simple
$H[\sym{G},\sym{E}]$-comodules that are pushed forward by $\bar\pi$ to
the \emph{same} simple $H$-comodule.

In the following, the assumption that $\sym{C}$ be left-autonomous,
\ie\ that $H=\coend(\sym{C},\omega)$ is a WHA, is not only used to
imply that the monoidal unit is multiplicity-free, but is rather a
key to comparing the categories of comodules of $H[\sym{G},\sym{E}]$
and of $H$.

The assumption guarantees that it is sufficient to determine the
preimage under $\hat\pi$ of the monoidal unit $\omega\one$. It turns
out that all simple isomorphism types in that preimage are obtained by
conjugating the monoidal unit of $\sym{M}^{H[\sym{G},\sym{E}]}$ by
group-like elements $g$ such that $\bar\pi(g)=1$. Dividing the
endomorphism adapted WBA $H[\sym{G},\sym{E}]$ by $1-g$ for these
group-likes $g$ then yields a WHA that is isomorphic to $H$.

\subsection{Group-like comodules}
\label{sect_gplike}

In this subsection, we consider an arbitrary WBA $H$.

\begin{definition}
An element $g\in H$ of a WBA $H$ is called
\begin{myenumerate}
\item
\emph{right group-like} if $\Delta g = g1^\prime\otimes g1^\pprime$ and $\epsilon_s(g)=1$,
\item
\emph{left group-like} if $\Delta g = 1^\prime g\otimes1^\pprime g$ and $\epsilon_t(g)= 1$,
\item
\emph{group-like} if it is both right and left group-like.
\end{myenumerate}
\end{definition}

The set of group-like elements in a WBA forms a monoid. Note that we
do not require the group-like elements of a WBA to have a
multiplicative inverse. If $H$ is a WHA, however, every group-like
$g\in H$ has the inverse $g^{-1}=S(g)$.

\begin{proposition}
\label{prop_gplike}
Let $H$ be a WBA.
\begin{myenumerate}
\item
For each group-like $g\in H$, there is a right $H$-comodule structure on
the vector space $H_s$ given by
\begin{equation}
\label{eq_gplikecomod}
\beta_{H_s^g}\colon H_s\to H_s\otimes H,\qquad x\mapsto x_0\otimes (gx_1),
\end{equation}
where $\beta_\one\colon H_s\to H_s\otimes H, x\mapsto x_0\otimes x_1$ denotes
the right $H$-comodule structure on $H_s$ that yields the monoidal unit of
$\sym{M}^H$, \cf~\eqref{eq_betahs}. In the following, we denote the comodule
with the structure~\eqref{eq_gplikecomod} by $H_s^g$.
\item
If $g_1,g_2\in H$ are group-like, then
\begin{equation}
H_s^{g_1}\hotimes H_s^{g_2}\cong H_s^{g_1g_2}.
\end{equation}
\end{myenumerate}
\end{proposition}

\begin{proof}
Parts~(1) and~(3) are straightforward. For part~(2), the left-unit constraint
of $H_s^{g_2}$, $\lambda_{H_s^{g_2}}\colon\one\hotimes H_s^{g_2}\to H_s^{g_2}$,
yields the linear map underlying the isomorphism $H_s^{g_1}\hotimes
H_s^{g_2}\to H_s^{g_1g_2}$.
\end{proof}

Given a homomorphism of bialgebras $f\colon H\to H^\prime$ and the
induced push-forward functor
$f^\ast\colon\sym{M}^H\to\sym{M}^{H^\prime}$, the $H$-comodules sent
by $f^\ast$ to the monoidal unit $\one\in|\sym{M}^{H^\prime}|$ are
precisely the comodules $H^g_s$ for the group-like elements $g\in H$
that satisfy $f(g)=1$. This is established in the remainder of this
subsection and will be applied to the homomorphism $\bar\pi\colon
H[\sym{G},\sym{E}]\to H$.

\begin{proposition}[from {\cite[Lemma~4.1]{Sc03}}]
In every WBA $H$, the restriction ${\epsilon_t}|_{H_s}\colon H_s\to
H_t$ forms an algebra anti-isomorphism with inverse
${\bar\epsilon_s}|_{H_t}$ where
$\bar\epsilon_s(x)=1^\prime\epsilon(1^\pprime x)$ for all $x\in H$.
\end{proposition}

\begin{proposition}[from {\cite[Lemma~6.3]{Sc03}}]
Let $f\colon H\to H^\prime$ be a homomorphism of WBAs. Then the
restrictions $f|_{H_s}\colon H_s\to H^\prime_s$ and $f|_{H_t}\colon
H_t\to H^\prime_t$ form isomorphisms of algebras. Their inverses are
${f|_{H_s}}^{-1}(y) = 1^\prime\epsilon(f(1^\pprime)y)$ for all $y\in
H^\prime_s$ and ${f|_{H_t}}^{-1}(z) = \epsilon(zf(1^\prime))1^\pprime$
for all $z\in H^\prime_t$, respectively.
\end{proposition}

The following proposition generalizes some results of
Nikshych~\cite{Ni02} and Vecserny{\'e}s~\cite{Ve03} from WHAs to WBAs.

\begin{proposition}
\label{prop_grouplikecomodule}
Let $H$ be a WBA.
\begin{myenumerate}
\item
There is a one-to-one correspondence between right group-like
elements $g\in H$ and right $H$-coactions $\beta\colon H_s\to
H_s\otimes H$ that satisfy
\begin{equation}
(\id_{H_s}\otimes\epsilon_s)\circ\beta = (\id_{H_s}\otimes\epsilon_s)\circ\Delta\circ\epsilon_s.
\end{equation}
The correspondence is given by $g=\epsilon(1_0)1_1$ and
$\beta(x)=\epsilon_s(x^\prime)\otimes gx^\pprime$ for all $x\in
H_s$.
\item
There is a one-to-one correspondence between left group-like
elements $g\in H$ and right $H$-coactions $\beta\colon H_t\to
H_t\otimes H$ that satisfy
\begin{equation}
(\id_{H_t}\otimes\epsilon_t)\circ\beta = (\epsilon_t\otimes\id_{H_t})\circ\Delta\circ\epsilon_t.
\end{equation}
The correspondence is given by $g=\epsilon(1_0)1_1$ and
$\beta(z)=\epsilon_t({(zg)}^\prime)\otimes{(zg)}^\pprime$.
\end{myenumerate}
\end{proposition}

\begin{theorem}
\label{thm_comodulepush}
Let $f\colon H\to H^\prime$ be a homomorphism of WBAs. Then the
induced functor $f^\ast\colon\sym{M}^H\to\sym{M}^{H^\prime}$ pushes
forward some right $H$-comodule $N\in|\sym{M}^H|$ to the monoidal unit
$\one\in|\sym{M}^{H^\prime}|$ if and only if $N\cong H_s^g$ for some
group-like $g\in H$ that satisfies $f(g)=1$. This group-like element
is given by
$g=(\epsilon\otimes\id_H)\circ\beta_N\circ\eta(1)=\epsilon(1_0)1_1\in
H$.
\end{theorem}

\begin{proof}
If $N\cong H_s^g$, then the coaction $\beta_N\colon H_s\to H_s\otimes
H$ is given by $\beta_N(x)=x^\prime\otimes(gx^\pprime)$ for all $x\in
H_s$. The push-forward reads $\beta_{f^\ast(N)}=(f|_{H_s}\otimes
f)\circ\beta_N\circ{f|_{H_s}}^{-1}\colon H^\prime_s\to
H^\prime_s\otimes H^\prime$. A direct computation shows that, since
$f(g)=1$, we have $\beta_{f^\ast(N)}(y) = y^\prime\otimes y^\pprime$
for all $y\in H^\prime_s$,
\ie\ $f^\ast(N)\cong\one\in|\sym{M}^{H^\prime}|$.

Conversely, let $f^\ast(N)\cong\one\in|\sym{M}^{H^\prime}|$, \ie\ there
is a coaction $\beta_{f^\ast(N)}\colon H^\prime_s\to H^\prime_s\otimes
H$ such that $y_0\otimes f(y_1)=y^\prime\otimes y^\pprime$ for all
$y\in H^\prime_s$.

First, the original coaction coincides with
$\beta_N=({f|_{H_s}}^{-1}\otimes\id_H)\circ\beta_{f^\ast(N)}\circ
f|_{H_s}\colon H_s\to H_s\otimes H$ which can be shown to satisfy
$x_0\otimes\epsilon_s(x_1)=x^\prime\otimes\epsilon_s(x^\pprime)$ for
all $x\in H_s$. By Proposition~\ref{prop_grouplikecomodule}(1),
$g=(\epsilon\otimes\id_H)\circ\beta_N\circ\eta(1)$ is right group-like
and $\beta_N(x)=x^\prime\otimes(g x^\pprime)$ for all $x\in H_s$.

Second, there is a coaction
$\delta=((\epsilon_t\circ{f|_{H_s}}^{-1})\otimes\id_H)\circ
f|_{H_s}\circ\bar\epsilon_s\colon H_t\to H_t\otimes H$ which can be
shown to satisfy
$z_0\otimes\epsilon_t(z_1)=\epsilon_t(z^\prime)\otimes z^\pprime$ for
all $z\in H_t$. By Proposition~\ref{prop_grouplikecomodule}(2),
$((\epsilon\circ\epsilon_t\circ{f|_{H_s}}^{-1})\otimes\id_H)\circ
f_{H_s}\circ\bar\epsilon_s\circ\eta(1)=g$ is left group-like as well.
\end{proof}

\subsection{Completing the characterization}
\label{sect_dividegplike}

In this subsection, we compute the kernel of the surjection
$\bar\pi\colon H[\sym{G},\sym{E}]\to H$ of
Proposition~\ref{prop_factorendo} and arrive at our characterization
of the universal coend $H=\coend(\sym{C},\omega)$ in
Theorem~\ref{thm_main}. The main technical result is the application
of Theorem~\ref{thm_comodulepush} to the homomorphism of WBAs
$\bar\pi\colon H[\sym{G},\sym{E}]\to H$.

\begin{theorem}
\label{thm_main}
Let $\sym{C}$ be an essentially small, finitely split semisimple,
$k$-linear, additive, autonomous monoidal category such that $k$ is a
field and $\Hom(X,Y)$ is finite-dimensional over $k$ for all
$X,Y\in|\sym{C}|$. We choose an object $M\in|\sym{C}|$ that generates
$\sym{C}$. Let $\sym{G}$ be the dimension graph and
$\sym{E}={(E^{(n)})}_{n\in\N_0}$ be an endomorphism system for
$\sym{C}$ with respect to $M$. We use the map $\bar\pi\colon
H[\sym{G},\sym{E}]\to H$ of Proposition~\ref{prop_factorendo}. Let $G$
be the set
\begin{equation}
\label{eq_setofgplikes}
G=\{\,g\in H[\sym{G},\sym{E}]\mid\quad\mbox{$g$ is group-like and}\quad\bar\pi(g)=1\,\}
\end{equation}
and $I_G$ be the two-sided ideal generated by the set $\{\,g-1\mid\, g\in
G\,\}$. Then $\bar\pi$ induces an isomorphism of WBAs
\begin{equation}
\label{eq_isomorphism}
H[\sym{G},\sym{E}]/I_G\cong H.
\end{equation}
\end{theorem}

\begin{proof}
If $g$ is group-like such that $\bar\pi(g)=1$, the two-sided ideal
generated by $g-1$ is also a two-sided coideal. The quotient $\tilde
H:=H[\sym{G},\sym{E}]/I_G$ is therefore a WBA. The map $\bar\pi\colon
H[\sym{G},\sym{E}]\to H$ obviously factors through this quotient and
yields another surjection of WBAs $\tilde\pi\colon\tilde H\to H$. We
have to show that this map $\tilde\pi$ is injective.

We know from Proposition~\ref{prop_schurweyl}(5) that
$H[\sym{G},\sym{E}]$ is a coproduct of matrix coalgebras. We first
show that $\tilde H$ is a coproduct of matrix coalgebras as well, and
then examine the action of $\tilde\pi$ on these matrix coalgebras in
order to establish the injectivity of $\tilde\pi$. In the following,
we denote by $p$ the canonical projection in the commutative diagram
\begin{equation}
\xymatrix{
H[\sym{G},\sym{E}]\ar[rr]^{p}\ar[ddrr]_{\bar\pi}&&\tilde H=H[\sym{G},\sym{E}]/I_G\ar[dd]^{\tilde\pi}\\
\\
&&H}
\end{equation}
Let $V\in|\sym{M}^{H[\sym{G},\sym{E}]}|$ be simple and
$C(V)=V^\ast\otimes V\subseteq H[\sym{G},\sym{E}]$ be the associated
matrix coalgebra. We know from the proof of
Proposition~\ref{prop_schurweyl}(5) that the restriction
$\bar\pi|_{C(V)}$ is injective, \ie\ $p|_{C(V)}$ is injective as
well. Since $p$ is surjective, $\tilde H$ is spanned by matrix
coalgebras of the form $p(C(V))$,
$V\in|\sym{M}^{H[\sym{G},\sym{E}]}|$.

Let now $W\in|\sym{M}^H|$ be simple and $C(W)\cong W^\ast\otimes
W\subseteq H$. By Proposition~\ref{prop_schurweyl}(5), its pre-image
$\bar\pi^{-1}(C(W))\subseteq H[\sym{G},\sym{E}]$ is a finite direct
sum of matrix coalgebras. For each of these matrix coalgebras
$C(V)=V^\ast\otimes V\subseteq H[\sym{G},\sym{E}]$, the restriction
$p|_{C(V)}$ is injective, and so $p(\bar\pi^{-1}(C(W))$ is a direct
sum of (perhaps a smaller number of) matrix coalgebras. Since $p$ is
surjective, $\tilde\pi^{-1}(C(W))=p(\bar\pi^{-1}(C(W)))$, \ie\ the
pre-image of $C(W)$ under $\tilde\pi$ is a finite direct sum of matrix
coalgebras.  In order to establish that $\tilde\pi$ is injective, it
therefore suffices to show that this finite direct sum consists of one
term only.

Recall that every potential term in the direct sum
$\tilde\pi^{-1}(C(W))$ is of the form $p(V^\ast\otimes V)$ for some
simple $V\in|\sym{M}^{H[\sym{G},\sym{E}]}|$, and so $V$ appears as a
subcomodule of some ${(\omega M)}^{\hotimes m}$, $m\geq 0$. We
therefore need to prove the following:

Let $X\subseteq{(\omega M)}^{\hotimes m}$ and $Y\subseteq{(\omega
M)}^{\hotimes\ell}$, $m,\ell\in\N_0$, be simple right
$H[\sym{G},\sym{E}]$-comodules. If $X\cong Y$ as $H$-comodules (under
push-forward by $\bar\pi$), then $X\cong Y$ as $\tilde H$-comodules
(under push-forward by $p$).

Let $X\cong Y$ be isomorphic as $H$-comodules. Recall that if
$m=\ell$, Proposition~\ref{prop_schurweyl}(4) implies that $X\cong Y$
as $H[\sym{G},\sym{E}]$-comodules and therefore as $\tilde
H$-comodules as well.  We still have to deal with the case
$m\neq\ell$.

Since $H$ is a WHA, $X^\ast\subseteq{(\omega M)}^{\hotimes t}$ for some
$t\in\N_0$. There is an $H$-comodule
\begin{equation}
T\subseteq X\hotimes X^\ast\subseteq{(\omega M)}^{\hotimes(m+t)},
\end{equation}
such that $T\cong H_s$ as $H$-comodules. Since $X\cong Y$ as $H$-comodules,
there is another $H$-comodule
\begin{equation}
\tilde T\subseteq Y\hotimes X^\ast\subseteq {(\omega M)}^{\hotimes(\ell+t)},
\end{equation}
such that $\tilde T\cong H_s$ as $H$-comodules. Both $T$ and $\tilde T$ are
also $H[\sym{G},\sym{E}]$-comodules that are pushed-forward under $\bar\pi$ to
the monoidal unit of $\sym{M}^H$. By Theorem~\ref{thm_comodulepush},
$T\cong{H[\sym{G},\sym{E}]}^g_s$ and $\tilde
T\cong{H[\sym{G},\sym{E}]}^{\tilde g}_s$ for some group-like $g,\tilde g\in
G$. Therefore, $T\cong\tilde H_s\cong\tilde T$ as $\tilde H$-comodules.

On the other hand,
\begin{eqnarray}
X\hotimes\tilde T &\subseteq& X\hotimes(Y\hotimes X^\ast)\subseteq{(\omega M)}^{\hotimes(m+\ell+t)},\\
Y\hotimes       T &\subseteq& Y\hotimes(X\hotimes X^\ast)\subseteq{(\omega M)}^{\hotimes(\ell+m+t)},
\end{eqnarray}
are isomorphic as $H$-comodules, and so by
Proposition~\ref{prop_schurweyl}(2), also as $H[\sym{G},\sym{E}]$-comodules
and therefore as $\tilde H$-comodules. We conclude that as $\tilde H$-comodules,
\begin{equation}
X\cong X\hotimes\tilde H_s\cong X\hotimes\tilde T
\cong Y\hotimes T\cong Y\hotimes \tilde H_s\cong Y.
\end{equation}
\end{proof}

Note that since $H$ is a WHA, the left hand side
of~\eqref{eq_isomorphism} becomes a WHA as well. Finally, we provide
some additional details of the construction.

\subsection{Some further details}

\begin{lemma}
\label{la_comparegplikes}
Under the assumptions of Theorem~\ref{thm_main}, let $g,\tilde g\in
H[\sym{G},\sym{E}]$ be group-like such that $\bar\pi(g)=1=\bar\pi(\tilde
g)$. Let ${H[\sym{G},\sym{E}]}^g_s\subseteq {(\omega M)}^{\hotimes m}$ and
${H[\sym{G},\sym{E}]}^{\tilde g}_s\subseteq {(\omega M)}^{\hotimes n}$ for
some $m,n\in\N_0$. If $m=n$, then $g=\tilde g$.
\end{lemma}

\begin{proof}
Recall from Proposition~\ref{prop_schurweyl}(5) that
$H[\sym{G},\sym{E}]$ is a coproduct of matrix coalgebras
$C(V)=V^\ast\otimes V$ each of which is associated with a subcomodule
$V\subseteq {(\omega M)}^{\hotimes m}$ for some $m\geq 0$, \ie\ all
comodules ${H[\sym{G},\sym{E}]}_s^g$ are subcomodules of the form assumed.

Let now $m=n$. Pushing-forward these comodules along $\bar\pi$ yields
isomorphic $H$-comodules $H^{\pi(g)}_s\cong H_s\cong H^{\pi(\tilde g)}_s$. By
Proposition~\ref{prop_schurweyl}(4),
${H[\sym{G},\sym{E}]}^g_s\cong{H[\sym{G},\sym{E}]}^{\tilde g}_s$ are also
isomorphic as $H[\sym{G},\sym{E}]$-comodules.

We now apply Theroem~\ref{thm_comodulepush} to the homomorphism of WBAs
$\bar\pi\colon H[\sym{G},\sym{E}]\to H$ and choose
$N={H[\sym{G},\sym{E}]}^{\tilde g}$. We know that $N$ pushes forward to the
monoidal unit of $\sym{M}^H$ under $\bar\pi$ and that it is isomorphic as an
$H[\sym{G},\sym{E}]$-comodule to ${H[\sym{G},\sym{E}]}_s^g$ where $g$ is
group-like with $\bar\pi(g)=1$. Theorem~\ref{thm_comodulepush} then gives a
formula for $g$ in terms of the coaction of ${H[\sym{G},\sym{E}]}_s^{\tilde
g}$:
\begin{equation}
g=(\epsilon\otimes\id)\circ\beta_{{H[\sym{G},\sym{E}]}_s^{\tilde
g}}\circ\eta(1)
=\epsilon (1^\prime)(\tilde g 1^\pprime)=\tilde g.
\end{equation}
\end{proof}

\begin{proposition}
Under the assumptions of Theorem~\ref{thm_main}, if in addition
$\sym{C}$ satisfies the strong Schur--Weyl property, then each $g\in
G$ is central in $H[\sym{G},\sym{E}]$ and satisfies $g=X(g)$ where
$X(g)=\epsilon_t(\epsilon_s(g^\prime))g^\pprime\epsilon_s(\epsilon_t(g^\pprime))$. Here,
$X$ acts on homogeneous elements ${[p|q]}_m\in
{H[\sym{G},\sym{E}]}_m$, $m\in\N_0$, as follows:
\begin{equation}
X({[p|q]}_m)=\delta_{\sigma(p),\tau(p)}\delta_{\sigma(q),\tau(q)}{[p|q]}_m.
\end{equation}
Furthermore, using the surjection $\pi_R\colon H[\sym{G},R]\to H$ of
WBAs of Proposition~\ref{prop_coquasi}, each $g\in G_R$ with
\begin{equation}
G_R=\{\,g\in H[\sym{G},R]\mid\quad\mbox{$g$ is group-like and}\quad\pi_R(g)=1\,\}
\end{equation}
is central in $H[\sym{G},R]$.
\end{proposition}

\begin{proof}
We abbreviate $N={H[\sym{G},\sym{E}]}_s$ and $N^g={H[\sym{G},\sym{E}]}^g_s$
and consider the two coactions
\begin{alignat}{2}
\beta_M&\colon M\to M\otimes H[\sym{G},\sym{E}],\quad
p\mapsto&\sum_{q\in\sym{G}^1}q\otimes{[q|p]}_1,\\
\beta_{N^g}&\colon N^g\to N^g\otimes H[\sym{G},\sym{E}],\quad
j\mapsto&\sum_{\ell\in\sym{G}^0}\ell\otimes(g{[\ell|j]}_0),
\end{alignat}
for $p\in\sym{G}^1$, $j\in\sym{G}^0$. Applying
$\id_{N^g}\otimes\id_M\otimes\bar\pi$ to both sides of the following equation
\begin{equation}
\beta_{N^g\hotimes M}(j\otimes p)
= (\sigma^{-1}_{N^g,M}\otimes\id_{H[\sym{G},\sym{E}]})
\circ\beta_{M\hotimes N^g}\circ\sigma_{N^g,M}(j\otimes p),
\end{equation}
exploiting that $\bar\pi(g)=1$ and that $\bar\pi$ is injective on
${H[\sym{G},\sym{E}]}_0\cdot{H[\sym{G},\sym{E}]}_1\subseteq{H[\sym{G},\sym{E}]}_1$
yields
\begin{equation}
\label{eq_central1}
\beta_{N\hotimes M}=(\sigma^{-1}_{N^g,M}\otimes\id_{H[\sym{G},\sym{E}]})\circ\beta_{M\hotimes N}
\circ\sigma_{N^g,M}
\end{equation}
where $\beta_{N\hotimes M}$ and $\beta_{M\hotimes N}$ contain the coaction of
the monoidal unit,
\begin{equation}
\beta_N\colon N^g\to N^g\otimes H[\sym{G},\sym{E}],\quad
j\mapsto\sum_{\ell\in\sym{G}^0}\ell\otimes{[\ell|j]}_0.
\end{equation}
Applying $(-\otimes-\otimes (g\cdot-))$ to~\eqref{eq_central1}, pre- and
post-composing with $\sigma^{-1}_{N^g,M}$ and
$\sigma_{N^g,M}\otimes\id_{H[\sym{G},\sym{E}]}$, and then pre- and
post-composing with $\rho_M$ and $\rho_M^{-1}$, respectively, allows us to
compute
\begin{equation}
\label{eq_centralstep}
\sum_{q\in\sym{G}^1}q\otimes(g{[q|p]}_1) = \sum_{q\in\sym{G}^1}q\otimes({[q|p]}_1X(g)).
\end{equation}

Now we repeat all of the above argument for
$\one={H[\sym{G},\sym{E}]}_s$ rather than $M$. In this case, we
exploit the fact that $\bar\pi$ is injective on
${H[\sym{G},\sym{E}]}_0\cdot{H[\sym{G},\sym{E}]}_0\subseteq{H[\sym{G},\sym{E}]}_0$
and obtain that for all $\ell\in\sym{G}^0$:
\begin{equation}
\sum_{j\in\sym{G}^0}j\otimes(g{[j|\ell]}_0) = \sum_{j\in\sym{G}^0}j\otimes({[j|\ell]}_0X(g)).
\end{equation}
Since $\sym{C}$ satisfies the strong Schur--Weyl property, $\one$ is
simple and therefore the ${[j|\ell]}_0$, $j,\ell\in\sym{G}^0$, form a basis
of ${(\omega M)}^{\hotimes 0}$, and so comparing coefficients yields
\begin{equation}
g{[j|\ell]}_0 = {[j|\ell]}_0X(g)
\end{equation}
for all $j,\ell\in\sym{G}^0$. Since
$\eta(1)\in{H[\sym{G},\sym{E}]}_0$, we conclude that $g=X(g)$.

Finally, since $M$ is simple, the ${[p|q]}_1$, $p,q\in\sym{G}^1$, form
a basis of $\omega M$, and so we can compare coefficients
in~\eqref{eq_centralstep} and find that
\begin{equation}
g{[p|q]}_1 = {[p|q]}_1g
\end{equation}
for all $p,q\in\sym{G}^1$. Since $H[\sym{G},\sym{E}]$ as an algebra is
generated by ${H[\sym{G},\sym{E}]}_1$, $g$ is central. The argument
for $G_R$ and $\pi_R$ is identical.

In order to compute $X(g)$ on homogeneous elements, we note that
\begin{eqnarray}
\epsilon_s({[p|q]}_m)
&=& \delta_{pq}\sum_{j\in\sym{G}^0} {[j|\sigma(p)]}_0,\\
\epsilon_t({[p|q]}_m)
&=& \delta_{pq}\sum_{j\in\sym{G}^0} {[\tau(q)|j]}_0,\\
\epsilon_t(\epsilon_s({[p|q]}_m)
&=& \delta_{pq}\sum_{j\in\sym{G}^0} {[\sigma(p)|j]}_0,\\
\epsilon_s(\epsilon_t({[p|q]}_m)
&=& \delta_{pq}\sum_{j\in\sym{G}^0} {[j|\tau(q)]}_0,
\end{eqnarray}
from which the claim follows in a direct computation.
\end{proof}

The following proposition is useful if one wishes to determine the
relevant group-like elements for a given category $\sym{C}$.

\begin{proposition}
\label{prop_gplikeunique}
Under the assumptions of Theorem~\ref{thm_main} with the set $G$
of~\eqref{eq_setofgplikes}, the following holds.
\begin{myenumerate}
\item
For each $m\in\N_0$, the set $G\cap{H[\sym{G},\sym{E}]}_m$ contains at
most one element which we denote by $g_m$ if it exists.
\item
For the elements of Part~(1), we have $g_mg_\ell=g_{m+\ell}$ for all $m,\ell$.
\item
$G$ is an abelian monoid.
\item
Each $g\in G$, $g\neq 1$, has infinite order.
\end{myenumerate}
\end{proposition}

\begin{proof}
\begin{myenumerate}
\item
Let $g,g^\prime\in G\cap{H[\sym{G},\sym{E}]}_m$ for some
$m\in\N_0$. Then by Theorem~\ref{thm_comodulepush} and by the
grading of $H[\sym{G},\sym{E}]$, both
${H[\sym{G},\sym{E}]}^g_s\subseteq{(\omega M)}^{\hotimes m}$ and
${H[\sym{G},\sym{E}]}^{g^\prime}_s\subseteq {(\omega M)}^{\hotimes
m}$. Lemma~\ref{la_comparegplikes} implies $g=g^\prime$.
\item
By the grading, $g_mg_\ell\in{H[\sym{G},\sym{E}]}_{m+\ell}$ and $g_mg_\ell$
is group-like with $\bar\pi(g_mg_\ell)=1$. Part~(1) then implies the claim.
\item
Because $\bar\pi$ is a homomorphism of unital associative algebras and
because of Part~(2).
\item
Because of Parts~(1) and~(2) and the fact that the unit is in degree
zero, $1\in{H[\sym{G},\sym{E}]}_0$.
\end{myenumerate}
\end{proof}

\begin{proposition}
Under the assumptions of Theorem~\ref{thm_main}, if $\sym{C}$ is
braided, the WBA $H[\sym{G},\sym{E}]/I_G$ is coquasi-triangular with
the universal $r$-form induced from $H[\sym{G},\sym{E}]$, and the
isomorphism~\eqref{eq_isomorphism} is an isomorphism of
coquasi-triangular WBAs.
\end{proposition}

\begin{proof}
The coquasi-triangular structure of $H[\sym{G},\sym{E}]$ descends to the
quotient $H[\sym{G},\sym{E}]/I_G$ provided that the universal $r$-form $r\colon
H[\sym{G},\sym{E}]\otimes H[\sym{G},\sym{E}]\to k$ and its weak convolution
inverse $\bar r\colon H[\sym{G},\sym{E}]\otimes H[\sym{G},\sym{E}]\to k$
satisfy
\begin{equation}
r(I_G,H[\sym{G},\sym{E}])=0=r(H[\sym{G},\sym{E}],I_G),\qquad
\bar r(I_G,H[\sym{G},\sym{E}])=0=\bar r(H[\sym{G},\sym{E}],I_G).
\end{equation}
This holds on the generators $1-g$, $g\in G$, of $I_G$ because $\bar\pi$ is a
homomorphism of coquasi-triangular WBAs, $\bar\pi(g)=1$ and
$r(1,-)=\epsilon(-)=r(-,1)$ and $\bar r(1,-)=\epsilon(-)=\bar r(-,1)$. It
extends to the two-sided ideal by~\eqref{eq_univr1} and~\eqref{eq_univr2}. The
isomorphism~\eqref{eq_isomorphism} is one of coquasi-triangular WHAs because
$\bar\pi$ pushes forward all relevant comodules.
\end{proof}

\section{Examples}
\label{sect_example}

\subsection{The modular categories associated with $U_q(\mathfrak{sl}_2)$}

In this section, we review the modular categories associated with
$U_q(\mathfrak{sl}_2)$ at suitable roots of unity following~\cite{KaLi94} and
present them as the categories of finite-dimensional comodules of a WHA
$H[\sym{G},\sym{E}]/I_G$.

\subsubsection{Description of the categories}

Let $r\in\{2,3,4,\ldots\}$ and $A$ be a primitive $4r$-th root of unity,
$q=A^2$. For simplicity, we work over the complex numbers $k=\C$. The
morphisms of our category $\sym{C}$ are represented by plane projections of
oriented framed tangles, drawn in blackboard framing. The coherence theorem
for ribbon categories~\cite{ReTu90} ensures that each diagram defines a
morphism of $\sym{C}$. Since $\sym{C}$ is $k$-linear, we can take formal
linear combinations of diagrams with coefficients in $k$. All our diagrams are
read from top to bottom.

The braiding of $\sym{C}$ is such that a crossing in our plane projections can
be resolved using the recursion relation for the Kauffman bracket
\begin{equation}
\begin{xy}
(-4,4);(4,-4) **\dir{-};
(-4,-4);(-1,-1) **\dir{-};
(1,1);(4,4) **\dir{-}
\end{xy}
= A\,
\begin{xy}
(-4,4);(4,4) **\crv{(-2,2)&(2,2)};
(-4,-4);(4,-4) **\crv{(-2,-2)&(2,-2)};
\end{xy}
+ A^{-1}\,
\begin{xy}
(4,-4);(4,4) **\crv{(2,-2)&(2,2)};
(-4,-4);(-4,4) **\crv{(-2,-2)&(-2,2)};
\end{xy}\,,\qquad\qquad
\begin{xy}
(0,0)*\xycircle<3mm,3mm>{-};
\end{xy}
= -(q+q^{-1}),
\end{equation}
ignoring the orientations for now. The Jones--Wenzl idempotents $P_n$,
$1\leq n\leq r-2$, are formal linear combinations of planar
$(n,n)$-tangles that can be defined recursively by
\begin{equation}
\begin{xy}
(0,0)*+{\hbox to 10mm{\strut\hfill$P_1$\hfill}}*\frm{-}!U="t" !D="b";
"t";(0, 12) **\dir{-};
"b";(0,-12) **\dir{-};
\end{xy}
:=
\begin{xy}
(0,-12);(0,12) **\dir{-};
\end{xy}\,,\qquad\qquad
\begin{xy}
(0,0)*+{\hbox to 10mm{\strut\hfill$P_{n+1}$\hfill}}*\frm{-}!U="t" !D="b";
"t"+<-4mm,0mm>;(-4, 12) **\dir{-};
"b"+<-4mm,0mm>;(-4,-12) **\dir{-};
"t"+<-3mm,0mm>;(-3, 12) **\dir{-};
"b"+<-3mm,0mm>;(-3,-12) **\dir{-};
"t"+<-2mm,0mm>;(-2, 12) **\dir{-}; ?(.5)+<3mm,0mm>*{\cdots};
"b"+<-2mm,0mm>;(-2,-12) **\dir{-}; ?(.5)+<3mm,0mm>*{\cdots};
"t"+< 4mm,0mm>;( 4, 12) **\dir{-};
"b"+< 4mm,0mm>;( 4,-12) **\dir{-};
\end{xy}
:=
\begin{xy}
(0,0)*+{\hbox to 10mm{\strut\hfill$P_n$\hfill}}*\frm{-}!U="t" !D="b";
"t"+<-4mm,0mm>;(-4, 12) **\dir{-};
"b"+<-4mm,0mm>;(-4,-12) **\dir{-};
"t"+<-3mm,0mm>;(-3, 12) **\dir{-};
"b"+<-3mm,0mm>;(-3,-12) **\dir{-};
"t"+<-2mm,0mm>;(-2, 12) **\dir{-}; ?(.5)+<3mm,0mm>*{\cdots};
"b"+<-2mm,0mm>;(-2,-12) **\dir{-}; ?(.5)+<3mm,0mm>*{\cdots};
"t"+< 4mm,0mm>;( 4, 12) **\dir{-};
"b"+< 4mm,0mm>;( 4,-12) **\dir{-};
(8,-12);(8,12) **\dir{-};
\end{xy}
+\frac{\q{n}}{\q{n+1}}\,
\begin{xy}
(0, 6)*+{\hbox to 10mm{\strut\hfill$P_n$\hfill}}*\frm{-}!U="t1" !D="b1";
(0,-6)*+{\hbox to 10mm{\strut\hfill$P_n$\hfill}}*\frm{-}!U="t2" !D="b2";
"t1"+<-4mm,0mm>;(-4, 12)        **\dir{-};
"b1"+<-4mm,0mm>;"t2"+<-4mm,0mm> **\dir{-};
"b2"+<-4mm,0mm>;(-4,-12)        **\dir{-};
"t1"+<-3mm,0mm>;(-3, 12)        **\dir{-}; ?(.5)+<2.5mm,0mm>*{\cdots};
"b1"+<-3mm,0mm>;"t2"+<-3mm,0mm> **\dir{-}; ?(.5)+<2.5mm,0mm>*{\cdots};
"b2"+<-3mm,0mm>;(-3,-12)        **\dir{-}; ?(.5)+<2.5mm,0mm>*{\cdots};
"t1"+< 3mm,0mm>;( 3, 12)        **\dir{-};
"b1"+< 3mm,0mm>;"t2"+< 3mm,0mm> **\dir{-};
"b2"+< 3mm,0mm>;( 3,-12)        **\dir{-};
"t1"+< 4mm,0mm>;( 4, 12)        **\dir{-};
"b2"+< 4mm,0mm>;( 4,-12)        **\dir{-};
"b1"+< 4mm,0mm>;( 8, 12) **\crv{"b1"+<4mm,-2mm>&"b1"+<8mm,-2mm>&"b1"+<8mm,0mm>};
"t2"+< 4mm,0mm>;( 8,-12) **\crv{"t2"+<4mm, 2mm>&"t2"+<8mm, 2mm>&"t2"+<8mm,0mm>};
\end{xy}\,.
\end{equation}
where $\q{n}=(q^n-q^{-n})/(q-q^{-1})$, $n\in\Z$, are the quantum integers. The
isomorphism classes of simple objects of $\sym{C}$ are indexed by the set
$I=\{0,1,\ldots,r-2\}$. The identity morphism of the object $V_n$, $n\in I$,
is the identity $(n,n)$-tangle with the idempotent $P_n$ inserted somewhere
(anywhere). As a shortcut, we write a single line labeled by $n$,
\begin{equation}
\begin{xy}
(0,-8);(0,8) **\dir{-} ?(.8)+<2mm,0mm>*{n};
\end{xy}
:=
\begin{xy}
(0,0)*+{\hbox to 10mm{\strut\hfill$P_n$\hfill}}*\frm{-}!U="t" !D="b";
"t"+<-4mm,0mm>;(-4, 8) **\dir{-};
"b"+<-4mm,0mm>;(-4,-8) **\dir{-};
"t"+<-3mm,0mm>;(-3, 8) **\dir{-};
"b"+<-3mm,0mm>;(-3,-8) **\dir{-};
"t"+<-2mm,0mm>;(-2, 8) **\dir{-}; ?(.5)+<3mm,0mm>*{\cdots};
"b"+<-2mm,0mm>;(-2,-8) **\dir{-}; ?(.5)+<3mm,0mm>*{\cdots};
"t"+< 4mm,0mm>;( 4, 8) **\dir{-};
"b"+< 4mm,0mm>;( 4,-8) **\dir{-};
\end{xy}\,.
\end{equation}
The object $V_0$ indexed by $0\in I$ is the monoidal unit and can be made
invisible in our diagrams thanks to the coherence theorem. The categorical
dimension of the simple objects is given by
\begin{equation}
\Delta_n :=
\begin{xy}
(0,0)*\xycircle<3mm,3mm>{-};
(3.5,3.5)*{n}
\end{xy}
= (-1)^n\q{n+1},
\end{equation}
which is non-zero for all $n\in I$.

Two special features of $U_q(\ssl_2)$ are exploited. First, the simple objects
are isomorphic to their duals, and the choice of representatives $V_j$, $j\in
I$, of the simple objects is such that ${(V_j)}^\ast=V_j$ are equal rather
than merely isomorphic. This allows us to omit any arrows from the diagrams
that would indicate the orientation of the ribbon tangle.

Second, there are no higher multiplicities, \ie\ for all $a,b,c\in I$,
we have $\dim_k\Hom(V_a\otimes V_b,V_c)\in\{0,1\}$. More precisely,
$\Hom(V_a\otimes V_b,V_c)\cong k$ if and only if the triple $(a,b,c)$
is \emph{admissible}. Otherwise, $\Hom(V_a\otimes V_b,V_c)=\{0\}$.

\begin{definition}
A triple $(a,b,c)\in I^3$ is called \emph{admissible} if the following
conditions hold.
\begin{myenumerate}
\item
$a+b+c\equiv 0$ mod $2$ (\emph{parity}),
\item
$a+b-c\geq 0$ and $b+c-a\geq 0$ and $c+a-b\geq 0$ (\emph{quantum triangle inequality}),
\item
$a+b+c\leq 2r-4$ (\emph{non-negligibility}).
\end{myenumerate}
\end{definition}

A special choice of basis vector of $\Hom(V_a,V_b\otimes V_c)$ is
denoted by a trivalent vertex:
\begin{equation}
\begin{xy}
(0,0)*\dir{*};
(0,0);(0,8) **\dir{-}; ?(.7)+<3mm,0mm>*{a};
(0,0);(-4,-8) **\crv{(-4,-4)&(-4,-6)}; ?(.7)+<-3mm,0mm>*{b};
(0,0);( 4,-8) **\crv{( 4,-4)&( 4,-6)}; ?(.7)+< 3mm,0mm>*{c};
\end{xy}
:=
\begin{xy}
(-2,3);( 2,3) **\dir{-};
( 2,3);( 2,5) **\dir{-};
( 2,5);(-2,5) **\dir{-};
(-2,5);(-2,3) **\dir{-};
( 0,5);( 0,8) **\dir{-}; ?(.5)+<3mm,0mm>*{a};
(-6,-5);(-2,-5) **\dir{-};
(-2,-5);(-2,-3) **\dir{-};
(-2,-3);(-6,-3) **\dir{-};
(-6,-3);(-6,-5) **\dir{-};
(-4,-5);(-4,-8) **\dir{-}; ?(.5)+<-3mm,0mm>*{b};
( 2,-5);( 6,-5) **\dir{-};
( 6,-5);( 6,-3) **\dir{-};
( 6,-3);( 2,-3) **\dir{-};
( 2,-3);( 2,-5) **\dir{-};
( 4,-5);( 4,-8) **\dir{-}; ?(.5)+<3mm,0mm>*{c};
(-1, 3);(-5,-3) **\crv{(-1,0)&(-5,0)}; ?(.5)+<-3mm,0mm>*{i};
( 1, 3);( 5,-3) **\crv{( 1,0)&( 5,0)}; ?(.5)+< 3mm,0mm>*{j};
(-3,-3);( 3,-3) **\crv{(-3,0)&( 3,0)}; ?(.5)+<0mm,-2mm>*{k};
\end{xy}\,,
\end{equation}
where $i=(a+b-c)/2$, $j=(a+c-b)/2$ and $k=(b+c-a)/2$ and the boxes
denote Jones--Wenzl idempotents. If we draw such a diagram for a
triple $(a,b,c)\in I^3$ that is not admissible, then by convention, we
multiply the entire diagram by zero. We also need the theta graph
\begin{equation}
\theta(a,b,c) :=
\begin{xy}
(-5,0)*\dir{*};
(5,0)*\dir{*};
(-5,0);(5,0) **\crv{(-5,5)&(5,5)}; ?(.5)+<0mm,2mm>*{a};
(-5,0);(5,0) **\dir{-}; ?(.5)+<0mm,2mm>*{b};
(-5,0);(5,0) **\crv{(-5,-5)&(5,-5)}; ?(.5)+<0mm,-2mm>*{c};
\end{xy}\,,
\end{equation}
which is non-zero for all admissible triples $(a,b,c)$. When we compose the
morphisms associated with such diagrams, the composition is zero unless the
labels at the open ends of the tangles match, \ie\ putting
\begin{equation}
\begin{xy}
(0,0)*\dir{*};
(0,0);(-4, 4) **\dir{-}; ?(.8)+<-3mm,0mm>*{r};
(0,0);( 4, 4) **\dir{-}; ?(.8)+< 3mm,0mm>*{j};
(0,0);( 0,-6) **\dir{-}; ?(.8)+< 3mm,0mm>*{s};
\end{xy}\qquad\mbox{below}\qquad
\begin{xy}
(0,0)*\dir{*};
(0,0);(-4,-4) **\dir{-}; ?(.8)+<-3mm,0mm>*{q};
(0,0);( 4,-4) **\dir{-}; ?(.8)+< 3mm,0mm>*{k};
(0,0);( 0, 6) **\dir{-}; ?(.8)+< 3mm,0mm>*{p};
\end{xy}\qquad\mbox{gives}\qquad
\delta_{qr}\delta_{kj}\,\,
\begin{xy}
(0, 4)*\dir{*};
(0,-4)*\dir{*};
(0,-4);(0, 4) **\crv{(-4,-4)&(-4, 4)}; ?(.5)+<-2mm,0mm>*{q};
(0,-4);(0, 4) **\crv{( 4,-4)&( 4, 4)}; ?(.5)+< 2mm,0mm>*{k};
(0, 4);(0, 8) **\dir{-}; ?(.8)+< 3mm,0mm>*{p};
(0,-4);(0,-8) **\dir{-}; ?(.8)+< 3mm,0mm>*{s};
\end{xy}
\end{equation}
Exploiting semi-simplicity and Schur's lemma, we compute
\begin{equation}
\begin{xy}
(0, 4)*\dir{*};
(0,-4)*\dir{*};
(0,-4);(0, 4) **\crv{(-4,-4)&(-4, 4)}; ?(.5)+<-2mm,0mm>*{q};
(0,-4);(0, 4) **\crv{( 4,-4)&( 4, 4)}; ?(.5)+< 2mm,0mm>*{k};
(0, 4);(0, 8) **\dir{-}; ?(.8)+< 3mm,0mm>*{p};
(0,-4);(0,-8) **\dir{-}; ?(.8)+< 3mm,0mm>*{s};
\end{xy}
=\delta_{ps}\frac{\theta(p,q,k)}{\Delta_p}\,
\begin{xy}
(0,-8);(0,8) **\dir{-}; ?(.8)+<3mm,0mm>*{p}
\end{xy}.
\end{equation}
The quantum $6j$-symbol is defined as
\begin{equation}
\left\{\begin{matrix}a&b&i\\ c&d&j\end{matrix}\right\}_q\,
:=\frac{\Delta_i}{\theta(a,d,i)\theta(b,c,i)}\,
\begin{xy}
( 0, 6)*\dir{*};
(-6, 0)*\dir{*};
( 6, 0)*\dir{*};
( 0,-6)*\dir{*};
(-6, 0);( 6, 0) **\dir{-}; ?(.5)+< 0mm,-2mm>*{j};
(-6, 0);( 0,-6) **\dir{-}; ?(.5)+<-2mm,-2mm>*{a};
(-6, 0);( 0, 6) **\dir{-}; ?(.5)+<-2mm, 2mm>*{b};
( 0, 6);( 6, 0) **\dir{-}; ?(.5)+< 2mm, 2mm>*{c};
( 0,-6);( 6, 0) **\dir{-}; ?(.5)+< 2mm,-2mm>*{d};
(12, 6);(12,-6) **\dir{-}; ?(.5)+< 3mm, 0mm>*{i};
( 0, 6);(12, 6) **\crv{(0,12)&(12,12)};
( 0,-6);(12,-6) **\crv{(0,-12)&(12,-12)};
\end{xy}\,.
\end{equation}
It is used in the recoupling identity,
\begin{equation}
\begin{xy}
(-4, 0)*\dir{*};
( 4, 0)*\dir{*};
(-4, 0);( 4, 0) **\dir{-}; ?(.5)+<0mm, 3mm>*{j};
( 4, 0);( 8, 4) **\dir{-}; ?(.8)+<0mm, 3mm>*{c};
( 4, 0);( 8,-4) **\dir{-}; ?(.8)+<0mm,-3mm>*{d};
(-4, 0);(-8,-4) **\dir{-}; ?(.8)+<0mm,-3mm>*{a};
(-4, 0);(-8, 4) **\dir{-}; ?(.8)+<0mm, 3mm>*{b};
\end{xy} = \sum_i\,
\left\{\begin{matrix}a&b&i\\ c&d&j\end{matrix}\right\}_q\,
\begin{xy}
(0,4)*\dir{*};
(0,-4)*\dir{*};
(0,-4);( 0, 4) **\dir{-}; ?(.5)+< 3mm,0mm>*{i};
(0, 4);(-4, 8) **\dir{-}; ?(.8)+<-3mm,0mm>*{b};
(0, 4);( 4, 8) **\dir{-}; ?(.8)+< 3mm,0mm>*{c};
(0,-4);(-4,-8) **\dir{-}; ?(.8)+<-3mm,0mm>*{a};
(0,-4);( 4,-8) **\dir{-}; ?(.8)+< 3mm,0mm>*{d};
\end{xy}\,.
\end{equation}

\subsubsection{The dimension graph}

We now assume that $r\geq 3$. The category $\sym{C}$ is generated by $M=V_1$
(Definition~\ref{def_generates}). From the decompositions $V_0\otimes M\cong
V_1$, $V_j\otimes M\cong V_{j-1}\oplus V_{j+1}$ for all $1\leq j\leq r-3$ and
$V_{r-2}\otimes M\cong V_{r-3}$, we obtain the dimension graph $\sym{G}$ of
$\sym{C}$ with respect to $M$ (Definition~\ref{def_graph}):
\begin{equation}
\label{eq_graphex1}
\sym{G} =
\begin{xy}
( 0, 0)*\dir{*};
( 0,-5)*{0};
( 1, 1);(14, 1)**\crv{( 5, 2)&(10, 2)} ?>*\dir{>};
( 1,-1);(14,-1)**\crv{( 5,-2)&(10,-2)} ?<*\dir{<};
(15, 0)*\dir{*};
(15,-5)*{1};
(16, 1);(29, 1)**\crv{(20, 2)&(25, 2)} ?>*\dir{>};
(16,-1);(29,-1)**\crv{(20,-2)&(25,-2)} ?<*\dir{<};
(30, 0)*\dir{*};
(30,-5)*{2};
(31, 1);(35, 1.5)**\crv{(33, 1.5)};
(31,-1);(35,-1.5)**\crv{(33,-1.5)} ?<*\dir{<};
(40, 0)*{\ldots};
(49, 1);(45, 1.5)**\crv{(47, 1.5)} ?<*\dir{<};
(49,-1);(45,-1.5)**\crv{(47,-1.5)};
(50, 0)*\dir{*};
(50,-5)*{r-3};
(51, 1);(64, 1)**\crv{(55, 2)&(60, 2)} ?>*\dir{>};
(51,-1);(64,-1)**\crv{(55,-2)&(60,-2)} ?<*\dir{<};
(65, 0)*\dir{*};
(65,-5)*{r-2};
\end{xy}
\end{equation}
Since for any two vertices $j,\ell\in\sym{G}^0=I$, there is at most
one edge from $j$ to $\ell$, we specify a path $p\in\sym{G}^m$ of
length $m\in\N_0$ by the sequence of the $m+1$ vertices along $p$,
\ie\ $p=(i_0,\ldots,i_m)\in I^{m+1}$. The source and target of this
path are $\sigma(p)=i_m$ and $\tau(p)=i_0$.

At this point, the reader should be familiar with the WBA $H[\sym{G}]$
associated with the graph $\sym{G}$ (Proposition~\ref{prop_defhg}). As
an algebra, $H[\sym{G}]\cong k(\sym{G}\times\sym{G})$ is the path
algebra of the quiver $\sym{G}\times\sym{G}$. As a coalgebra, it is a
direct sum of matrix coalgebras: one for each degree, \ie\ for each
length of paths. Our construction shows that the category $\sym{C}$ is
equivalent to the category of finite-dimensional comodules of
$H[\sym{G}]$ modulo the relations~\eqref{eq_relex11}
and~\eqref{eq_relex12} below.

\subsubsection{The fundamental surjection}

We use the same basis of $\omega M=\Hom(\hat V,\hat V\otimes M)$ as
in~\cite{Pf09a,Pf09b},
\ie\ $\{\,e_{i,i+1}^{(M)},e_{i+1,i}^{(M)}\,\mid\,0\leq i\leq r-3\,\}$
which reads in terms of diagrams,
\begin{equation}
e^{(M)}_{i,i+1} =
\begin{xy}
(0,0)*\dir{*};
(0,0);(0,6) **\dir{-}; ?(.7)+<3mm,0mm>*{i};
(0,0);(-4,-4) **\dir{-}; ?(.7)+<-6mm,-1mm>*{i+1};
(0,0);( 4,-4) **\dir{-}; ?(.7)+< 3mm,-1mm>*{1};
\end{xy}\qquad\mbox{and}\qquad
e^{(M)}_{i+1,i} =
\begin{xy}
(0,0)*\dir{*};
(0,0);(0,6) **\dir{-}; ?(.7)+<6mm,0mm>*{i+1};
(0,0);(-4,-4) **\dir{-}; ?(.7)+<-3mm,-1mm>*{i};
(0,0);( 4,-4) **\dir{-}; ?(.7)+< 3mm,-1mm>*{1};
\end{xy}.
\end{equation}
The surjection $\pi\colon H[\sym{G}]\to H=\coend(\sym{C},\omega)$
(Theorem~\ref{thm_surjection}) then maps
\begin{eqnarray}
\pi({[(j)|(\ell)]}_0) &=& {[e^j_{(\one)}|e^{(\one)}_\ell]}_\one,\\
\pi({[(j_0,j_1)|(\ell_0,\ell_1)]}_1)
&=& {[e^{j_0,j_1}_{(M)}|e^{(M)}_{\ell_1,\ell_0}]}_M,
\end{eqnarray}
for $j,j_0,j_1,\ell,\ell_0,\ell_1\in I$ and $j_1=j_0\pm 1$, $\ell_1=\ell_0\pm
1$. We refer to the explanations preceding Theorem~\ref{thm_surjection} for
the bases used on the right hand side.

\subsubsection{The endomorphism adapted WBA}

The category $\sym{C}$ satisfies the strong Schur--Weyl property
(Definition~\ref{def_schurweyl}), \ie\ the monoidal unit $\one=V_0$
and the generating object $M=V_1$ are both simple, and the
endomorphism algebras $\End(M^{\otimes m})$, $m\geq 2$, are generated
by braiding and inverse braiding of adjacent tensor factors.

The ideal $I_\sym{E}$ in the definition of the endomorphism adapted WBA
$H[\sym{G},\sym{E}]=H[\sym{G}]/I_\sym{E}$
(Definitions~\ref{def_schurweyladapted} and~\ref{def_weakfrt}) is therefore
generated by homogeneous relations of degree two:
\begin{eqnarray}
\label{eq_relex11}
& & \sum_{(i_0,i_1,i_2)\in\sym{G}^2}
{[(j_0,j_1,j_2)|(i_0,i_1,i_2)]}_2\,R_{(i_0,i_1,i_2);(\ell_0,\ell_1,\ell_2)}\nn\\
&-& \sum_{(i_0,i_1,i_2)\in\sym{G}^2}
R_{(j_0,j_1,j_2);(i_0,i_1,i_2)}\,{[(i_0,i_1,i_2)|(\ell_0,\ell_1,\ell_2)]}_2,
\end{eqnarray}
for all paths of length two $(j_0,j_1,j_2)\in\sym{G}^2$ and
$(\ell_0,\ell_1,\ell_2)\in\sym{G}^2$. Note that these relations are
non-trivial only if $j_0=\ell_0$ and $j_2=\ell_2$
(Remark~\ref{rem_coefficients}(2)).

A direct computation using Temperley--Lieb recoupling calculus yields
the following non-zero coefficients:
\begin{eqnarray}
R_{(j,j\pm1,j);(j,j\pm1,j)}
&=& \mp q^{-1/2}\,\frac{q^{\pm(j+1)}}{\q{j+1}},\\
R_{(j,j-1,j);(j,j+1,j)}
&=& q^{-1/2}\frac{\q{j}\q{j+2}}{{\q{j+1}}^2},\\
R_{(j,j+1,j);(j,j-1,j)}
&=& q^{-1/2},\\
R_{(j,j\pm1,j\pm2);(j,j\pm1,j\pm2)} &=& q^{-3/2}.
\end{eqnarray}

\subsubsection{The relevant group-like elements}

In order to compute the kernel $I_G$ of the induced surjection
$\bar\pi\colon H[\sym{G},\sym{E}]\to H$, we systematically consider
the simple comodules of $H$ and $H[\sym{G},\sym{E}]$, proceeding by
increasing degree according to the tensor power $m$ of $M^{\otimes
m}$, $m\in\N_0$.

In the following table, we show the decomposition of ${(\omega
M)}^{\hotimes m}$ as an $H$-comodule which is known from $\sym{C}$,
and the decomposition of $k\sym{G}^m$ as an
$H[\sym{G},\sym{E}]$-comodule which follows from
Proposition~\ref{prop_schurweyl}(4). Assume for now that $r$ is big.

\begin{center}
\begin{tabular}{r|l|l}
$m$ & ${(\omega M)}^{\hotimes m}\in|\sym{M}^H|$ & $k\sym{G}^m\in|\sym{M}^{H[\sym{G},\sym{E}]}|$\\
\hline
$0$ & $V_0$ & $V_0$\\
$1$ & $V_1$ & $V_1$\\
$2$ & $V_2\oplus V_0$ & $V_2\oplus V_0^\prime$\\
$3$ & $V_3\oplus 2V_1$ & $V_3\oplus 2V_1^\prime$\\
$4$ & $V_4\oplus 3V_2\oplus 2V_0$ & $V_4\oplus 3V_2^\prime\oplus 2V_0^{\pprime}$\\
$\ldots$ & $\ldots$ & $\ldots$
\end{tabular}
\end{center}

For each $j\in $I, the objects $V_j$, $V_j^\prime$, $V_j^\pprime$,
$\ldots$ are simple $H[\sym{G},\sym{E}]$-comodules that are pairwise
non-isomorphic as $H[\sym{G},\sym{E}]$-comodules but that are all
pushed forward to the $H$-comodule $V_j$ under $\bar\pi$. Note that we
have suppressed the long forgetful functor and written $V_j$ for
$\omega V_j$.

We see that $m=2$ is the smallest degree in which there is an
$H[\sym{G},\sym{E}]$-comodule, $V_0^\prime$, which is not isomorphic
to $V_0$, but pushed forward to it under $\bar\pi$. Therefore, by
Theorem~\ref{thm_comodulepush}, $V_0^\prime$ is characterized by a
group-like element $g_2\in{H[\sym{G},\sym{E}]}_2$ for which
$\bar\pi(g_2)=1$. Upon dividing $H[\sym{G},\sym{E}]$ by the relation
$g_2-1$, $V_0^\prime$ and $V_0$ will become isomorphic.

The next higher degree with an $H[\sym{G},\sym{E}]$-comodule
non-isomorphic, but pushed forward to $V_0$ is $m=4$. Since the
group-like $g_2^2$ is of degree $m=4$ and satisfies $\bar\pi(g_2)=1$,
by Proposition~\ref{prop_gplikeunique}, $g_2^2$ is the group-like that
characterizes $V_0^\pprime$. Notice that the quotient by $g_2-1$ will
also render $V_1^\prime$ isomorphic to $V_1$, $V_0^\pprime$ isomorphic
to $V_0$ and $V_2^\prime$ isomorphic to $V_2$.

If $r$ is not large enough, the above argument is unchanged except
that some of the `biggest' comodules are absent from the
decompositions. The pattern continues in higher degrees, and the only
relevant group-like is $g_2$.

In order to compute $g_2$, we explicitly decompose $\omega
M\hotimes\omega M\cong V_2\oplus V_0^\prime$ and compute $g_2$ as the
group-like that characterizes $V_0^\prime$ from
Theorem~\ref{thm_comodulepush}. For the decomposition, we
calculate the idempotent $Q=\omega(\id_{M\otimes M}-P_2)$ associated
with $V_0^\prime$. Here, $P_2$ is the Jones--Wenzl idempotent of
$V_2\subseteq M\otimes M$. $Q$ takes non-zero values in the following
cases:
\begin{eqnarray}
Q((0,1,0)) &=& (0,1,0),\\
Q((j,j+1,j))
&=& \frac{\q{j+2}}{\q{2}\q{j+1}}\,(j,j+1,j)
-\frac{\q{j}\q{j+2}}{\q{2}{\q{j+1}}^2}\,(j,j-1,j),\\
Q((j,j-1,j))
&=& -\frac{1}{\q{2}}\,(j,j+1,j)
+\frac{\q{j}}{\q{2}\q{j+1}}\,(j,j-1,j),\\
Q((r-2,r-3,r-2)) &=& (r-2,r-3,r-2),
\end{eqnarray}
for $1\leq j\leq r-3$. A basis for $V_0^\prime\subseteq\omega
M\hotimes\omega M$ is therefore given by
\begin{eqnarray}
b_0 &=& (0,1,0),\\
b_j &=& \bigl(\q{j+1}\,(j,j+1,j)-\q{j}\,(j,j-1,j)\bigr)/\sqrt{2},\\
b_{r-2} &=& -(r-2,r-3,r-2).
\end{eqnarray}
We use the coefficients of the comodule $V_0^\prime$ in this basis and
compute the group-like to be
\begin{equation}
\begin{split}
\label{eq_qdet}
g_2=\sum_{j,\ell=0}^{r-2}\alpha_j\alpha_\ell\,\Biggl(
&\frac{\q{\ell+1}}{\q{j+1}}\,{[(j,j+1,j)|(\ell,\ell+1,\ell)]}_2
+\frac{\q{\ell}}{\q{j}}\,{[(j,j-1,j)|(\ell,\ell-1,\ell)]}_2\\
-&\frac{\q{\ell+1}}{\q{j}}\,{[(j,j-1,j)|(\ell,\ell+1,\ell)]}_2
-\frac{\q{\ell}}{\q{j+1}}\,{[(j,j+1,j)|(\ell,\ell-1,\ell)]}_2
\Biggr),
\end{split}
\end{equation}
with $\alpha_0=\alpha_{r-2}=1$ and $\alpha_j=1/\sqrt{2}$ for all
$1\leq j\leq r-3$. In~\eqref{eq_qdet} it is understood that terms
with a path $(j,j\pm1,j)$ are omitted from the expression whenever
$j\pm1<0$ or $j\pm1>r-2$.

The kernel of $\bar\pi\colon H[\sym{G},\sym{E}]\to H$ is therefore
(Theorem~\ref{thm_main}) generated by
\begin{equation}
\label{eq_relex12}
1-g_2.
\end{equation}
We have shown that the category $\sym{C}$ is equivalent to the category of
finite-dimensional comodules of the quotient of $H[\sym{G}]$ for the graph
$\sym{G}$ of~\eqref{eq_graphex1} modulo the relations~\eqref{eq_relex11}
and~\eqref{eq_relex12}.

\appendix
\section{Summary of notation and conventions}
\label{app_monoidal}

In this appendix, we collect the relevant definitions and properties of
monoidal, autonomous, braided monoidal and abelian categories,
following Schauenburg~\cite{Sc92} and MacLane~\cite{Ma73}.

\subsection{Monoidal categories}

\begin{definition}
A \emph{monoidal category}
$(\sym{C},\otimes,\one,\alpha,\lambda,\rho)$ is a category $\sym{C}$
with a bifunctor $\otimes\colon\sym{C}\times\sym{C}\to\sym{C}$
(\emph{tensor product}), an object $\one\in|\sym{C}|$ (\emph{monoidal
unit}) and natural isomorphisms $\alpha_{X,Y,Z}\colon(X\otimes
Y)\otimes Z\to X\otimes(Y\otimes Z)$ (\emph{associator}),
$\lambda_X\colon \one\otimes X\to X$ (\emph{left-unit constraint}) and
$\rho_X\colon X\otimes\one\to X$ (\emph{right-unit constraint}) for
all $X,Y,Z\in|\sym{C}|$, subject to the pentagon axiom
\begin{equation}
\alpha_{X,Y,Z\otimes W}\circ\alpha_{X\otimes Y,Z,W}
= (\id_X\otimes\alpha_{Y,Z,W})\circ\alpha_{X,Y\otimes Z,W}\circ(\alpha_{X,Y,Z}\otimes\id_W)
\end{equation}
and the triangle axiom
\begin{equation}
\rho_X\otimes\id_Y=(\id_X\otimes\lambda_Y)\circ\alpha_{X,\one,Y}
\end{equation}
for all $X,Y,Z,W\in|\sym{C}|$.
\end{definition}

\begin{definition}
\label{def_lax}
Let $(\sym{C},\otimes,\one,\alpha,\lambda,\rho)$ and
$(\sym{C}^\prime,\otimes^\prime,\one^\prime,\alpha^\prime,\lambda^\prime,\rho^\prime)$
be monoidal categories.
\begin{myenumerate}
\item
A \emph{lax monoidal functor}
$(F,F_{X,Y},F_0)\colon\sym{C}\to\sym{C}^\prime$ consists of a
functor $F\colon\sym{C}\to\sym{C}^\prime$, morphisms $F_{X,Y}\colon
FX\otimes^\prime FY\to F(X\otimes Y)$ that are natural in
$X,Y\in|\sym{C}|$, and of a morphism $F_0\colon\one^\prime\to
F\one$, subject to the hexagon axiom
\begin{equation}
F_{X,Y\otimes Z}\circ(\id_{FX}\otimes^\prime F_{Y,Z})\circ\alpha^\prime_{FX,FY,FZ}
= F\alpha_{X,Y,Z}\circ F_{X\otimes Y,Z}\circ(F_{X,Y}\otimes^\prime\id_{FZ})
\end{equation}
and the two squares
\begin{eqnarray}
\lambda^\prime_{FX} &=& F\lambda_X\circ F_{\one,X}\circ(F_0\otimes^\prime\id_{FX}),\\
\rho^\prime_{FX}    &=& F\rho_X\circ F_{X,\one}\circ(\id_{FX}\otimes^\prime F_0)
\end{eqnarray}
for all $X,Y,Z\in|\sym{C}|$.
\item
An \emph{oplax monoidal functor}
$(F,F^{X,Y},F^0)\colon\sym{C}\to\sym{C}^\prime$ consists of a
functor $F\colon\sym{C}\to\sym{C}^\prime$, morphisms $F^{X,Y}\colon
F(X\otimes Y)\to FX\otimes^\prime FY$ that are natural in
$X,Y\in|\sym{C}|$, and of a morphism $F^0\colon
F\one\to\one^\prime$, subject to the hexagon axiom
\begin{equation}
\label{eq_oplaxhexagon}
(\id_{FX}\otimes^\prime F^{Y,Z})\circ F^{X,Y\otimes Z}\circ F\alpha_{X,Y,Z}
= \alpha^\prime_{FX,FY,FZ}\circ(F^{X,Y}\otimes^\prime\id_{FZ})\circ F^{X\otimes Y,Z}
\end{equation}
and the two squares
\begin{eqnarray}
F\lambda_X &=& \lambda^\prime_{FX}\circ(F^0\otimes^\prime\id_{FX})\circ F^{\one,X},\\
F\rho_X    &=& \rho^\prime_{FX}\circ(\id_{FX}\otimes^\prime F^0)\circ F^{X,\one}
\end{eqnarray}
for all $X,Y,Z\in|\sym{C}|$.
\item
A \emph{strong monoidal functor}
$(F,F_{X,Y},F_0)\colon\sym{C}\to\sym{C}^\prime$ is a lax monoidal
functor such that all $F_{X,Y}$, $X,Y\in|\sym{C}|$ and $F_0$ are
isomorphisms.
\end{myenumerate}
\end{definition}

\begin{definition}
Let $(\sym{C},\otimes,\one,\alpha,\lambda,\rho)$ be a monoidal
category. A \emph{left-dual} $(X^\ast,\ev_X,\coev_X)$ of an object
$X\in|\sym{C}|$ consists of an object $X^\ast\in|\sym{C}|$ and morphisms
$\ev_X\colon X^\ast\otimes X\to\one$ (\emph{left evaluation}) and
$\coev_X\colon\one\to X\otimes X^\ast$ (\emph{left coevaluation}) that satisfy
the triangle identities
\begin{eqnarray}
\label{eq_zigzag1}
\rho_X\circ(\id_X\otimes\ev_X)\circ\alpha_{X,X^\ast,X}
\circ(\coev_X\otimes\id_X)\circ\lambda_X^{-1} &=& \id_X,\\
\label{eq_zigzag2}
\lambda_{X^\ast}\circ(\ev_X\otimes\id_{X^\ast})\circ\alpha^{-1}_{X^\ast,X,X^\ast}
\circ(\id_{X^\ast}\otimes\coev_X)\circ\rho_{X^\ast}^{-1} &=& \id_{X^\ast}.
\end{eqnarray}
If $\sym{C}$ is a monoidal category and $f\colon X\to Y$ a morphism of
$\sym{C}$ such that both $X$ and $Y$ have left-duals, the \emph{left-dual} of
$f$ is defined as
\begin{equation}
f^\ast := \lambda_{X^\ast}\circ(\ev_Y\otimes\id_{X^\ast})\circ\alpha^{-1}_{Y^\ast,Y,X^\ast}
\circ(\id_{Y^\ast}\otimes(f\otimes\id_{X^\ast}))\circ(\id_{Y^\ast}\otimes\coev_X)\circ\rho^{-1}_{Y^\ast}.
\end{equation}
A \emph{left-autonomous category} is a monoidal category in
which each object is equipped with a specified left-dual.
\end{definition}

\begin{definition}
A \emph{braided monoidal category}
$(\sym{C},\otimes,\one,\alpha,\lambda,\rho,\sigma)$ is a monoidal
category $(\sym{C},\otimes,\one,\alpha,\lambda,\rho)$ with natural
isomorphisms $\sigma_{X,Y}\colon X\otimes Y\to Y\otimes X$ for all
$X,Y\in|\sym{C}|$ that satisfy the two hexagon axioms
\begin{eqnarray}
\sigma_{X\otimes Y,Z} &=& \alpha_{Z,X,Y}\circ(\sigma_{X,Z}\otimes\id_Y)
\circ\alpha^{-1}_{X,Z,Y}\circ(\id_X\otimes\sigma_{Y,Z})\circ\alpha_{X,Y,Z},\\
\sigma_{X,Y\otimes Z} &=& \alpha^{-1}_{Y,Z,X}\circ(\id_Y\otimes\sigma_{X,Z})
\circ\alpha_{Y,X,Z}\circ(\sigma_{X,Y}\otimes\id_Z)\circ\alpha^{-1}_{X,Y,Z}
\end{eqnarray}
for all $X,Y,Z\in|\sym{C}|$. The category is called \emph{symmetric
monoidal} if in addition
\begin{equation}
\sigma_{Y,X}\circ\sigma_{X,Y} = \id_{X\otimes Y}
\end{equation}
for all $X,Y\in|\sym{C}|$.
\end{definition}

\begin{definition}
Let $(\sym{C},\otimes,\one,\alpha,\lambda,\rho,\sigma)$ and
$(\sym{C}^\prime,\otimes^\prime,\one^\prime,\alpha^\prime,\lambda^\prime,\rho^\prime,\sigma^\prime)$
be braided monoidal categories. A lax monoidal functor
$(F,F_{X,Y},F_0)\colon\sym{C}\to\sym{C}^\prime$ is called
\emph{braided} if
\begin{equation}
\label{eq_braidedfunctor}
F\sigma_{X,Y}\circ F_{X,Y} = F_{Y,X}\circ\sigma^\prime_{FX,FY}
\end{equation}
for all $X,Y\in|\sym{C}|$.
\end{definition}

\subsection{Abelian and semisimple categories}
\label{app_abelian}

\begin{definition}
A category $\sym{C}$ is called \emph{$\mathbf{Ab}$-enriched} if it is enriched
in the category $\mathbf{Ab}$ of abelian groups, \ie\ if $\Hom(X,Y)$
is an abelian group for all objects $X,Y\in|\sym{C}|$ and if the
composition of morphisms is $\Z$-bilinear.

Let $k$ be a commutative ring. A category $\sym{C}$ is called
\emph{$k$-linear} if it is enriched in ${}_k\sym{M}$, the category of
$k$-modules, \ie\ if $\Hom(X,Y)$ is a $k$-module for all
$X,Y\in|\sym{C}|$ and if the composition of morphisms is $k$-bilinear.

A functor $F\colon\sym{C}\to\sym{C}^\prime$ between [$\mathbf{Ab}$-enriched,
$k$-linear] categories is called [\emph{additive}, $k$-\emph{linear}]
if it induces homomorphisms of [additive groups, $k$-modules]
\begin{equation}
\Hom(X,Y)\to\Hom(FX,FY)
\end{equation}
for all $X,Y\in|\sym{C}|$.
\end{definition}

\begin{definition}
\label{def_preadditivecat}
A monoidal category $(\sym{C},\otimes,\one,\alpha,\lambda,\rho)$ is
called [\emph{$\mathbf{Ab}$-enriched}, $k$-\emph{linear}] if $\sym{C}$ is
[$\mathbf{Ab}$-enriched, $k$-linear] and if the tensor product of morphisms is
[$\Z$-bilinear, $k$-bilinear].
\end{definition}

\begin{definition}
An \emph{additive category} is an $\mathbf{Ab}$-enriched category that
has a terminal object and all binary products. A \emph{preabelian
category} is an $\mathbf{Ab}$-enriched category that has all finite
limits. An \emph{abelian category} is a preabelian category in which
every monomorphism is a kernel and in which every epimorphism is a
cokernel.
\end{definition}

\begin{definition}
\label{def_semisimple}
Let $\sym{C}$ be a $k$-linear category, $k$ a field.
\begin{myenumerate}
\item
An object $X\in|\sym{C}|$ is called \emph{simple} if $\End(X)\cong
k$ are isomorphic as $k$-modules.
\item
The category $\sym{C}$ is called \emph{split semisimple} if there exists a
family ${\{V_j\}}_{j\in I}$ of objects $V_j\in|\sym{C}|$, $I$ some
index set, such that
\begin{myenumerate}
\item
$V_j$ is simple for all $j\in I$.
\item
$\Hom(V_j,V_\ell)=\{0\}$ for all $j,\ell\in I$ for which $j\neq\ell$.
\item
For each object $X\in|\sym{C}|$, there is a finite sequence
$j_1^{(X)},\ldots,j_{n^X}^{(X)}\in I$, $n^X\in\N_0$, and morphisms
$\imath_\ell^{(X)}\colon V_{j_\ell}\to X$ and $\pi_\ell^{(X)}\colon
X\to V_{j_\ell}$ such that
\begin{equation}
\id_X = \sum_{\ell=1}^{n^X}\imath^X_\ell\circ\pi^X_\ell.
\end{equation}
and
\begin{equation}
\label{eq_semisimpledualbases}
\pi^X_\ell\circ{\imath}^X_m = \left\{
\begin{matrix}
\id_{V_{j^X_\ell}},&\mbox{if}\quad \ell=m,\\
0,                 &\mbox{else}
\end{matrix}
\right.
\end{equation}
\end{myenumerate}
\item
The category is called \emph{finitely split semisimple} if it is split
semisimple with a finite index set $I$ in condition~(3).
\end{myenumerate}
\end{definition}

When we speak of split semisimple monoidal categories, we do not require the
monoidal unit $\one$ to be simple. Note that in a split semisimple autonomous
category, if an object $X$ is simple, then so is its dual.

\subsubsection*{Acknowledgements}

The author is grateful to Gabriella B{\"o}hm, Catharina Stroppel,
Korn{\'e}l Szlach{\'a}nyi and Peter Vecserny{\'e}s for stimulating
discussions, to Vladimir Turaev for correspondence, and to the
anonymous referee for the hint that Theorem~\ref{thm_comodulepush}
holds for all homomorphisms between WBAs, a result which substantially
simplified Section~5.

\newenvironment{hpabstract}{%
\renewcommand{\baselinestretch}{0.2}
\begin{footnotesize}%
}{\end{footnotesize}}%
\newcommand{\hpeprint}[2]{%
\href{http://www.arxiv.org/abs/#1}{\texttt{arxiv:#1#2}}}%
\newcommand{\hpspires}[1]{%
\href{http://www.slac.stanford.edu/spires/find/hep/www?#1}{SPIRES Link}}%
\newcommand{\hpmathsci}[1]{%
\href{http://www.ams.org/mathscinet-getitem?mr=#1}{\texttt{MR #1}}}%
\newcommand{\hpdoi}[1]{%
\href{http://dx.doi.org/#1}{\texttt{DOI #1}}}%
\newcommand{\hpjournal}[2]{%
\href{http://dx.doi.org/#2}{\textsl{#1\/}}}%

\end{document}